\documentclass[aop,preprint]{imsart}
\usepackage{exscale,amsmath,amssymb,amsthm,color,graphicx,accents,mathrsfs,mathtools,hyperref}
\usepackage{enumitem}
\usepackage{multicol,multirow,hhline,comment,url}
\usepackage{cellspace}		
\startlocaldefs

\newcommand{\eq}{\begin{equation}}
\newcommand{\en}{\end{equation}}

\numberwithin{equation}{section}
\numberwithin{figure}{section}
\numberwithin{table}{section}

\allowdisplaybreaks[3]

\setlist[enumerate,1]{leftmargin=1cm}
\newcommand{\cev}[1]{\accentset{\leftharpoonup}{#1}}

\theoremstyle{plain}
\newtheorem{theorem}{Theorem}[section]
\newtheorem{proposition}[theorem]{Proposition}
\newtheorem{corollary}[theorem]{Corollary}
\newtheorem{lemma}[theorem]{Lemma}
\newtheorem{conjecture}{Conjecture}

\theoremstyle{definition}
\newtheorem{definition}[theorem]{Definition}

\newcommand{\IPLT}{\mathscr{D}}
\newcommand{\infsimp}{{\nabla}_{\!\infty}}

\newcommand{\petrovgen}{\mathcal{B}}
\newcommand{\procvgen}{\mathcal{A}}
\newcommand{\EV}{\mathbf{E}}	
\newcommand{\downto}{\downarrow}
\newcommand{\ltwo}{\mathbf{L}^2}

\newcommand{\norm}[1]{\left\lVert #1 \right\rVert}
\newcommand{\IPmag}[1]{\left\|\vphantom{I}#1\right\|}

\renewcommand{\Pr}{\mathbf{P}}	
\newcommand{\Prm}{\mathrm{P}}
\newcommand{\BPr}{\mathbb{P}}		
\newcommand{\BR}{\mathbb{R}}	

\newcommand{\concat}{\star}
\newcommand{\Concat}{ \mathop{ \raisebox{-2pt}{\Huge$\star$} } }

\makeatletter
 \DeclareRobustCommand{\checkarg}{\@ifnextchar[{\@witharg}{}}
 \DeclareRobustCommand{\@witharg}[1][]{\ensuremath{\left(#1\right)}}
 
 \DeclareRobustCommand{\scaleGen}[1]{\@ifnextchar[{\@scalewithargs{#1}}{\odot^{}_{#1}}}
 \def\@scalewithargs#1[#2][#3]{#2 \odot^{}_{#1} #3}
\makeatother

\newcommand{\distribfont}[1]{{\tt #1}}

\newcommand{\PoiDir}{\distribfont{PD}\checkarg}
\newcommand{\EKP}{\distribfont{EKP}\checkarg}
\newcommand{\GammaDist}{\distribfont{Gamma}\checkarg}
\newcommand{\PDIP}{\distribfont{PDIP}\checkarg}

\endlocaldefs

\begin{document}

\begin{frontmatter}

\title{Ranked masses in two-parameter Fleming--Viot diffusions\thanksref{T0}}

\runtitle{Ranked masses in ${\tt FV}(\alpha,\theta)$}
\thankstext{T0}{This research is partially supported by NSF grants {DMS-1204840, DMS-1308340, DMS-1444084, DMS-1612483, DMS-1855568}, UW-RRF grant A112251, NSERC RGPIN-2020-06907, and EPSRC grant EP/K029797/1}

\begin{aug}
 \author{\fnms{Noah} \snm{Forman}\thanksref{a}\ead[label=e1]{noah.forman@gmail.com}},
\author{\fnms{Soumik} \snm{Pal}\thanksref{b}\ead[label=e2]{soumikpal@gmail.com}},
\author{\fnms{Douglas} \snm{Rizzolo}\thanksref{c}\ead[label=e3]{drizzolo@udel.edu}} and 
\author{\fnms{Matthias} \snm{Winkel}\thanksref{d}\ead[label=e4]{winkel@stats.ox.ac.uk}}


\address[a]{Department of Mathematics \& Statistics, McMaster University, 1280 Main Street West, Hamilton, Ontario L8S 4K1, Canada,
   \printead{e1}
  }

\address[b]{Department of Mathematics, University of Washington, Seattle WA 98195, USA,
\printead{e2}
}

\address[c]{Department of Mathematical Sciences, University of Delaware, Newark DE 19716, USA, 
\printead{e3}
}

\address[d]{Department of Statistics, University of Oxford, 24--29 St Giles', Oxford OX1 3LB, UK, 
\printead{e4}
} 

\runauthor{N.~Forman, S.~Pal, D.~Rizzolo and M.~Winkel}
\end{aug} 
  
 \begin{abstract}
  In previous work, we constructed Fleming--Viot-type measure-valued diffusions (and diffusions on a space of interval partitions of the unit interval $[0,1]$) that are stationary with the Poisson--Dirichlet laws with parameters $\alpha\in(0,1)$ and $\theta\ge 0$. In this paper, we complete the
  proof that these processes resolve a conjecture by Feng and Sun (2010) by showing that the processes of ranked atom sizes (or of ranked interval lengths) of these diffusions are members of a two-parameter family of diffusions introduced by Petrov (2009), extending a model by Ethier and Kurtz (1981) in the case $\alpha=0$. The latter diffusions are continuum limits of up-down Chinese restaurant processes. 
 \end{abstract}



\end{frontmatter}

\section{Introduction}
\label{sec:intro}
Fleming--Viot processes corresponding to Petrov's \cite{Petrov09} two-parameter extension of Ethier and Kurtz infinitely-many-neutral-alleles diffusion model  \cite{EthiKurt81} were recently constructed in \cite{PaperFV, ShiWinkel-2}.  The existence of these Fleming--Viot processes was conjectured in \cite{FengSun10}.  Efforts to construct them through analytic approaches such as Dirichlet forms encountered significant challenges that have not yet been overcome, though progress has been made \cite{FengSun10,FengSun19, ZHANG201997}.  In contrast to the analytic approaches, these processes are constructed path-wise  in \cite{PaperFV, ShiWinkel-2} using Poisson random measures.  
While these constructions give easy access to a number of interesting properties of the Fleming--Viot processes, particularly sample path properties, deriving an analytic characterization of the processes remains challenging.  

In this paper, we initiate the study of the analytic properties of the processes constructed in \cite{PaperFV, ShiWinkel-2}.  Our main result is that we identify the evolution of the sequence of ranked atom masses in a Fleming--Viot process with parameters $0<\alpha<1$, $\theta>-\alpha$ as the corresponding diffusion constructed by Petrov.  Specifically, we identify this evolution as a diffusion on the Kingman simplex\vspace{-0.2cm}
\begin{equation}\label{eq:infsimp_def}
 \infsimp := \bigg\{\mathbf{x}=\left(x_1, x_2, \ldots   \right)\colon x_1 \ge x_2 \ge \cdots \ge 0,\; \sum_{i\ge 1} x_i = 1     \bigg\}\vspace{-0.2cm}
\end{equation}
with the following pre-generator acting on the unital algebra $\mathcal{F}$ of symmetric functions generated by $q_m(\mathbf{x})=\sum_{i\ge 1}x_i^{m+1}$, $m\ge 1$: \vspace{-0.2cm}
\begin{equation}\label{petrovgen}
\petrovgen=\sum_{i\ge 1}x_i\frac{\partial^2}{\partial x_i^2}-\sum_{i,j\ge 1}x_ix_j\frac{\partial^2}{\partial x_i\partial x_j}-\sum_{i\ge 1}(\theta x_i+\alpha)\frac{\partial}{\partial x_i}.\vspace{-0.2cm}
\end{equation}
The one-parameter family of diffusions with $\alpha=0$ is due to Ethier and Kurtz \cite{EthiKurt81}, while the extension to two parameters is due to Petrov \cite{Petrov09}. We denote the laws of this two-parameter family by $\EKP(\alpha,\theta)=(\EKP_{\mathbf{x}}(\alpha, \theta),\,\mathbf{x}\in\nabla_{\!\infty})$. The stationary distribution of 
${\tt EKP}(\alpha,\theta)$ is the well-known two-parameter Poisson--Dirichlet distribution ${\tt PD}(\alpha,\theta)$. Analogously, the stationary distributions of the Fleming--Viot processes are Poisson--Dirichlet random measures ${\tt PDRM}(\alpha,\theta)$, where $\overline{\Pi}=\sum_{i\ge 1}P_i\delta(U_i)$ with independent $(P_i,\,i\ge 1)\sim{\tt PD}(\alpha,\theta)$ and $U_i\sim{\tt Unif}[0,1]$, $i\ge 1$. This family of random measures is also known as the Dirichlet or Pitman--Yor process in Bayesian Statistics \cite{JLP08,Teh06}. 

The two-parameter $\EKP[\alpha,\theta]$ processes have been widely studied over the last decade using a variety of methods including Dirichlet forms, generators and discrete approximation. Up-down Chinese restaurant processes, finite particle models and finite-dimensional diffusions have been studied with an aim, in particular, to understand the mechanism by which the downward drift that is present in all coordinates of $\nabla_{\!\infty}$ is compensated by the creation of new parts (new allelic types in a genetic interpretation \cite{EthiKurt81}). See \cite{CdBERS17,Eth14,FengSun10,FengSunWangXu11,RivRiz20,Ruggiero14,RuggWalk09,RuggWalkFava13}.  The construction of Fleming--Viot processes corresponding to $\EKP(\alpha,\theta)$ diffusions was one of the longest-standing open problems in the field. As noted in \cite{PaperFV}, identifying the evolution of ranked atom masses of the process constructed there with the $\EKP(\alpha,\theta)$ diffusion completes the argument that those processes are the associated Fleming--Viot processes. Since
the mechanism by which new types are created by the Fleming--Viot process is explicit in the construction, this also explains how types are created in the ${\tt EKP}(\alpha,\theta)$ diffusion.

Our arguments are related to those that appear in the study of polynomial processes \cite{Cuchiero2012, eberlein2019mathematical}, which have recently drawn significant interest in mathematical finance for their balance of generality and computational tractability.  Recall that a (classical) polynomial process is a Markov process on $\BR^d$ whose semigroup preserves, for each $m$, the set of polynomials of degree at most $m$.  Recently there have also been efforts to extend the study of polynomial processes to the infinite-dimensional setting \cite{cuchiero2019,CS19,EDK20}, where the appropriate notion of ``polynomial'' depends on the context.  Jacobi diffusions and Wright--Fisher diffusions are two classical examples of polynomial processes, and the key step in our argument is to identify statistics of the Fleming--Viot processes constructed in \cite{PaperFV, ShiWinkel-2} that evolve as Jacobi diffusions and Wright--Fisher diffusions.  Using the action of the generators of these diffusions on a class of symmetric polynomials, we are able to compute the generator of the ranked sequence of atom masses in the Fleming--Viot processes.  A consequence of our calculations is that these Fleming--Viot processes are $\mathcal{F}$-polynomial processes in the sense of \cite{EDK20}.  We conjecture that they are polynomial processes in the sense of \cite{cuchiero2019,CS19}, but have thus far been unable to compute the generator on the polynomials considered there.


Let us define the measure-valued processes of \cite{PaperFV} that we called two-parameter Fleming--Viot processes ${\tt FV}(\alpha,\theta)$ in the parameter range $0<\alpha<1$ and $\theta\geq 0$. In this parameter range, the construction can be done in two steps.  We first use explicit transition kernels identified in \cite{PaperFV} to define purely-atomic-measure-valued self-similar superprocesses ${\tt SSSP}(\alpha,\theta)$ with a branching property. For the second step, we apply Shiga's \cite{Shiga1990} time-change/normalization, which we call de-Poissonization. Specifically, the \em branching property \em of ${\tt SSSP}(\alpha,\theta)$ means that each atom evolves independently (in size) and generates further atoms during its lifetime. De-Poissonization destroys the independence of the branching property.  A similar approach applies in the full parameter range, but explicit transition kernels of ${\tt SSSP}(\alpha,\theta)$ are unknown in this case, making the construction and analysis of the ${\tt FV}(\alpha,\theta)$ processes more complicated.  Nonetheless, our proofs for the $\theta\geq 0$ case apply with only minor changes.  Thus, for the sake of simplicity, we carry out our construction and proofs first when $\theta\geq 0$ and then indicate what changes must be made in the general case.

To construct the transition kernel for an ${\tt SSSP}(\alpha,\theta)$ process when $0<\alpha<1$ and $\theta\geq 0$, there are three cases for the masses arising at a later time $s>0$ from a given atom in the initial state: (i) this atom survives to time $s$, as do infinitely many descendant atoms; (ii) the atom does not survive, but its descendants do; or (iii) neither the atom nor its descendants survive. The transitions of ${\tt SSSP}(\alpha,\theta)$ can thus be described via a probabilistic mixture of these three cases, independently for each atom.

Rather than separate out these cases entirely, we combine cases (i) and (ii), which yields a nicer formula for the law of the mass of either the surviving initial atom or one of its descendants. For this purpose, consider random variables $L_{b,r}^{(\alpha)}$ with Laplace transforms
\begin{equation}
 \EV\left[e^{-\lambda L_{b,r}^{(\alpha)}}\right]=\left(\frac{r+\lambda}{r}\right)^\alpha\frac{e^{br^2/(r+\lambda)}-1}{e^{br}-1}, \quad \lambda\ge0,\ b>0,\ r>0.\label{eqn:LMB}
\end{equation}
For an atom location $u\!\in\![0,1]$, also consider a new location $U_0\!\sim\!{\tt Unif}[0,1]$ and mixing probabilities
$$p_{b,r}^{(\alpha)}(c)=\frac{I_{1+\alpha}(2r\sqrt{bc})}{I_{-1-\alpha}(2r\sqrt{bc})+\alpha(2r\sqrt{bc})^{-1-\alpha}/\Gamma(1-\alpha)}
\quad\mbox{and}\quad 1-p_{b,r}^{(\alpha)}(c),
$$
where $I_v$ is the modified Bessel function of the first kind of index $v\in\mathbb{R}$. 
Independently of $L_{b,r}^{(\alpha)}$ and $U_0$, consider $\overline{\Pi}\sim{\tt PDRM}(\alpha,\alpha)$ and $G\!\sim\!\GammaDist[\alpha,r]$ to define the random measure $\Pi:=G\overline{\Pi}$ of random mass $G$. 
Then
\begin{align}\nonumber
 Q_{b,u,r}^{(\alpha)}:=e^{-br}\delta_0+\big(1\!-\!e^{-br}\big)\!\!&\int_0^\infty\!\!\Big(p_{b,r}^{(\alpha)}(c)\Pr\!\left\{c\delta(u)\!+\!\Pi\!\in\cdot\right\}\\[-0.1cm]
\label{eq:Qbxralpha}				&\quad+\big(1\!-\!p_{b,r}^{(\alpha)}(c)\big)\Pr\!\left\{c\delta(U_0)\!+\!\Pi\!\in\cdot\right\}\!\!\Big)\Pr\big\{L_{b,r}^{(\alpha)}\!\!\in\! dc\big\}.
\end{align}

\noindent is the distribution of a random measure that we will use to generate descendants at time $s=1/2r$ for an initial atom $b\delta(u)$. We remark that $\delta_0$ refers to a Dirac mass at the zero measure, $0$, in the space of measures on $[0,1]$, whereas $\delta(u)$ and $\delta(U_0)$ are Dirac masses in the interval $[0,1]$. 

More precisely: $Q_{b,u,r}^{(\alpha)}$ yields no descendants with probability $e^{-br}$ (case (iii)); or else the atoms of $\Pi$ are descendants of the initial atom, and there is one additional atom of size $L_{b,r}^{(\alpha)}$. 
This special atom is located either at allelic type $u$ (in case (i)) with conditional probability $p_{b,r}^{(\alpha)}(c)$ given $L_{b,r}^{(\alpha)} = c$, or at $U_0$ otherwise (in case (ii)). In this last case, the atom $c\delta(U_0)$ is an additional descendant of $b\delta(u)$.
\begin{definition}[Transition kernel $K^{\alpha,\theta}_s$]\label{def:kernel:sp}
 Let $\alpha\in(0,1)$ and $\theta\ge 0$. For a time $s>0$ and a finite measure $\mu =\sum_{i\ge 1} b_i \delta(u_i)$ with distinct $u_i$, $i\ge 1$, we consider the 
 random measure $G_0\overline{\Pi}_0+\sum_{i\ge 1}\Pi_i$ for independent $G_0\sim\GammaDist[\theta,1/2s]$, $\overline{\Pi}_0\sim{\tt PDRM}(\alpha,\theta)$ and $\Pi_i\sim Q_{b_i,u_i,1/2s}^{(\alpha)}$, $i\ge 1$. We denote its distribution by $K_s^{\alpha,\theta}(\mu,\,\cdot\,)$.
\end{definition}
We showed in \cite[Theorem 1.2]{PaperFV}, that $(K_s^{\alpha,\theta}\!,\,s\!\ge\! 0)$ is the transition semi-group of a path-continuous measure-valued Hunt process, which we refer to as ${\tt SSSP}_\mu(\alpha,\theta)$ when starting from any finite purely atomic measure $\mu_0=\mu$ on $[0,1]$. To obtain a probability-measure-valued process, we considered $(\mu_s,\,s\ge 0)\sim{\tt SSSP}_\pi(\alpha,\theta)$ starting from any purely atomic probability measure $\mu_0=\pi$, its total mass process
$\|\mu_s\|:=\mu_s([0,1])$ and the time-change
\begin{equation}\label{eq:timechangeFV}\rho(t)=\inf\left\{s\ge 0\colon\int_0^s\frac{dv}{\|\mu_v\|}>t\right\},\quad t\ge 0.
\end{equation}
We called $\pi_t\!:=\!\|\mu_{\rho(t)}\|^{-1}\mu_{\rho(t)}$, $t\!\ge\! 0$, a ${\tt FV}_\pi(\alpha,\theta)$ and showed in \cite[Theorem 1.7]{PaperFV} that it is a Hunt process in the space $\mathcal{M}_1^a$ of purely atomic probability measures on $[0,1]$ and, moreover, it has as its stationary distribution ${\tt PDRM}(\alpha,\theta)$.

For a probability measure $\pi\!=\!\sum_{i\ge 1}p_i\delta(u_i)$ with $p_1\!\ge\! p_2\!\ge\!\cdots$, we denote by ${\tt RANKED}(\pi)\!:=(p_i,\,i\ge 1)\in\nabla_{\!\infty}$ its ranked sequence of atom sizes. The main result of this paper is the following connection to Petrov's $\nabla_{\!\infty}$-valued diffusions with Poisson--Dirichlet stationary distributions. 

The construction above is for $\theta\ge 0$. We postpone the extension to $\theta\in(-\alpha,0)$ to Section \ref{sec:negtheta}, but state our main results here in full generality.

\begin{theorem}\label{thm:petroviden} Let $\alpha\in(0,1)$, $\theta>-\alpha$ and $\pi\in\mathcal{M}_1^a$. For $(\pi_t,\,t\ge 0)\sim{\tt FV}_\pi(\alpha,\theta)$, we have 
  $\big({\tt RANKED}(\pi_{t/2}),\,t\ge 0\big)\sim{\tt EKP}_{{\tt RANKED}(\pi)}(\alpha,\theta)$.
\end{theorem}

This shows that ${\tt FV}(\alpha,\theta)$ may be viewed as a labeled variant of $\EKP[\alpha,\theta]$. Hence, the processes constructed in \cite{PaperFV,ShiWinkel-2} indeed solve the open problem of Feng and Sun \cite{FengSun10}. 

This theorem allows us to prove a number of properties of $\EKP(\alpha,\theta)$ processes based on our understanding of Fleming--Viot processes. For example, 
following \cite[equations (82) and (83)]{Pitman03}, for $\alpha\in (0,1)$, the \emph{$\alpha$-diversity} is
\begin{equation}
 \IPLT_{\alpha}(\mathbf{x}) := \lim_{h\downto 0} \Gamma(1-\alpha)h^{\alpha}\#\{i\geq 1\colon x_i>h\} \qquad \text{for }\mathbf{x}\in\nabla_{\!\infty},\label{eq:intro:diversity}
\end{equation}
if this limit exists. This may be understood as a continuum analogue to the number of blocks in a partition of $n$. A constant multiple of this is sometimes called the \emph{local time} of $\mathbf{x}$ \cite[equation (24)]{PitmYorPDAT}. These quantities arise in a variety of contexts \cite{CSP,PitmWink09}. Ruggiero et\ al.\ \cite{RuggWalkFava13} have studied processes related to $\EKP$ diffusions for which $\alpha$-diversity evolves as a diffusion. Then
$$\IPLT_{\alpha}(\pi):=\lim_{h\downto 0}\Gamma(1-\alpha)h^{\alpha}\#\{u\in[0,1]\colon\pi\{u\}>h\}=\IPLT_{\alpha}({\tt RANKED}(\pi)),$$
in the sense that either neither limit exists or they are equal. Since the path-continuity of $t\mapsto\IPLT_{\alpha}(\pi_t)$ was shown in \cite[Theorem 1.12]{PaperFV}, our Theorem \ref{thm:petroviden} here immediately implies the following.

\begin{corollary}\label{cor:Petrov_diversity} Let $\mathbf{x}\!\in\!\nabla_{\!\infty}$. Suppose the limit $ \IPLT_{\alpha}(\mathbf{x})$ in \eqref{eq:intro:diversity} exists.
  Then $\mathbf{V}\!\sim\!\EKP_\mathbf{x}(\alpha,\theta)$ a.s.\ has a continuous diversity process $t\!\mapsto\!\IPLT_{\alpha}(\mathbf{V}_t)$. 
\end{corollary}

Since $\EKP$ diffusions are reversible, the evolving ranked sequence of atom sizes in ${\tt FV}(\alpha,\theta)$ is reversible as well. We make the following conjecture.

\begin{conjecture}\label{conj:reverse_0}
 ${\tt FV}(\alpha,\theta)$ is reversible with respect to ${\tt PDRM}(\alpha,\theta)$.
\end{conjecture}

On the other hand, there is a loss of symmetry in the corresponding interval-partition-valued diffusions (except when $\theta=\alpha$), which we recall in the appendix. This means that reversibility fails for those diffusions. 

We prove Theorem \ref{thm:petroviden} in two steps. 
The first is to calculate relevant parts of the generator of ${\tt FV}(\alpha,\theta)$, which allows us to identify the semi-groups of $\EKP[\alpha,\theta]$ and ranked ${\tt FV}(\alpha,\theta)$ on the Hilbert space $\mathbf{L}^2[\alpha,\theta]$ of functions on $\nabla_{\!\infty}$ that are square integrable with respect to the measure ${\tt PD}(\alpha,\theta)$. The second step involves interval partition evolutions \cite{Paper1-1,Paper1-2,IPPAT,ShiWinkel-1,ShiWinkel-2}, which can be coupled with ${\tt FV}(\alpha,\theta)$-processes to have the same ranked masses. We use these couplings to establish sufficient regularity to identify the semi-groups also as operators acting on the space of bounded continuous functions.

The structure of this paper is as follows. In Section \ref{sec:prel} we collect some material about ${\tt SSSP}(\alpha,\theta)$ and ${\tt FV}(\alpha,\theta)$ from \cite{PaperFV} and strengthen connections to Jacobi and Wright--Fisher diffusions that will facilitate the generator calculations of ${\tt FV}(\alpha,\theta)$. In Section \ref{sec:Petrov}, we carry out the first step in the proof of Theorem \ref{thm:petroviden}. In Section \ref{sec:IP}, we obtain a version of Theorem \ref{thm:petroviden} for the interval partition evolutions of \cite{IPPAT,ShiWinkel-1} and use it to carry out the second step in the proof of Theorem \ref{thm:petroviden}.

\section{Fleming--Viot, Jacobi and Wright--Fisher processes}\label{sec:prel} This section recalls material on ${\tt BESQ}_b(2r)$, ${\tt JAC}_b(r,r^\prime)$ and ${\tt WF}_{\mathbf{b}}(\mathbf{r})$ processes and their connections due to Warren and Yor \cite{WarrYor98} and Pal \cite{Pal13}. In particular, we discuss the domains of their infinitesimal generators. Finally, we recall from \cite{PaperFV} some more details about the construction and properties of ${\tt SSSP}(\alpha,\theta)$ and ${\tt FV}(\alpha,\theta)$, and we go beyond \cite{PaperFV} by extracting from ${\tt FV}(\alpha,\theta)$ several ``subprocesses'' that are Jacobi diffusions or Wright--Fisher processes.

\subsection{Squared Bessel processes ${\tt BESQ}_b(2r)$}

Let $b\ge 0$, $r\in\mathbb{R}$, and consider a Brownian motion $B$. The squared Bessel process is the unique strong solution of
$$dZ_s=2rds+2\sqrt{|Z_s|}dB_s,\quad Z_0=b,$$
see \cite{GoinYor03,RevuzYor}. For $r\ge 0$, we denote its distribution by ${\tt BESQ}_b(2r)$. These processes are $[0,\infty)$-valued and have
0 as an inaccessible boundary for $r\ge 1$, a reflecting boundary for $0<r<1$ and are absorbed at 0 for $r=0$. For $r<0$, the strong solution
becomes negative after $S=\inf\{s\ge 0\colon Z_s=0\}$ and we denote by ${\tt BESQ}_b(2r)$ the distribution of the absorbed process
$(Z_{s\wedge S},\,s\ge 0)$. The infinitesimal generator of ${\tt BESQ}(2r)$ is
$$2z\frac{d^2}{dz^2}+2r\frac{d}{dz}$$
on a domain that includes all twice continuously differentiable functions $f\colon[0,\infty)\rightarrow\mathbb{R}$ with compact support in $(0,\infty)$. 

\subsection{Jacobi diffusions ${\tt JAC}_b(r,r^\prime)$}

Let $r,r^\prime\!\ge\!0$, $b\!\in\![0,1]$. Warren and Yor \cite{WarrYor98} take independent $Z\!\sim\!{\tt BESQ}_b(2r)$, $Z^\prime\!\sim\!{\tt BESQ}_{1-b}(2r^\prime)$ and the time-change
\begin{equation}\label{eq:timechangeJAC}\rho(t)=\inf\left\{s\ge 0\colon\int_0^s\frac{dv}{Z_v+Z^\prime_v}>t\right\},\quad t\ge 0.
\end{equation}
The time-changed proportion $X:=\big((Z_{\rho(t)}+Z^\prime_{\rho(t})^{-1}Z_{\rho(t)},\,t\ge 0\big)$ is shown to be a $[0,1]$-valued Markov process, a Jacobi diffusion \cite{Kimura1964}, which we denote by ${\tt JAC}_b(r,r^\prime)$. Jacobi diffusions satisfy the SDE
$$dX_t=2\sqrt{X_t(1-X_t)}dB_t+2\big(r-(r+r^\prime)X_t\big)dt,\quad X_0=b,$$
and have infinitesimal generator 
\begin{equation}\label{eq:JACgen} \mathcal{A}_{\tt JAC}^{r,r^\prime}=2x(1-x)\frac{d^2}{dx^2}+2\big(r-(r+r^\prime)x\big)\frac{d}{dx}.
\end{equation}
With some care at the boundaries of $[0,1]$, this all extends to general $r,r^\prime\in\mathbb{R}$ up to the time $S$ (or $S^\prime$ or $S\wedge S^\prime$) if $r<0$ (or $r^\prime<0$ or both), when our definition of ${\tt BESQ}$ leads, after time-change, to the absorption of 
${\tt JAC}_b(r,r^\prime)$ in 0 (or in 1 or in either). We do, however, only absorb at $0$ if $r\le 0$ and at $1$ if $r^\prime\le 0$, allowing reflection when $0<r<1$ or $0<r^\prime<1$, respectively. The respective boundary is inaccessible for $r\ge 1$ or $r^\prime\ge 1$.

\begin{lemma}\label{lm:domJAC}
For all $r,r'\in\BR$, the domain of $\mathcal{A}^{r,r^\prime}_{\tt JAC}$ includes all twice continuously differentiable functions $f$ on $[0,1]$ that further 
satisfy $f^\prime(0)=0$ if $r<0$ and $f^\prime(1)=0$ if $r^\prime<0$.
\end{lemma}
%
%
\begin{proof} Let $b\!\in\!(0,1)$. As for the squared Bessel SDE, the Jacobi SDE has a unique strong solution up to the absorption time $S$, 
  which is infinite if $r\!>\!0$ and $\!r^\prime\!>0$, exhibiting reflection at 0 and/or 1 if $0\!<\!r\!<\!1$ and/or $0\!<\!r^\prime\!<1$. By the (local) It\^o 
  formula and a change of variables,
$$\mathbf{E}[f(X_{s})]=f(b)+s\mathbf{E}\left[\int_0^{1\wedge S/s}\!\left(2(r\!-\!(r\!+\!r^\prime)X_{us})f^\prime(X_{us})+2X_{us}(1\!-\!X_{us})f^{\prime\prime}(X_{us})\right)du\right]$$
  for $X\sim{\tt JAC}_b(r,r^\prime)$. We conclude by path-continuity and dominated convergence that 
$$s^{-1}\big(\mathbf{E}[f(X_s)]-f(b)\big)\rightarrow 2(r\!-\!(r\!+\!r^\prime)b)f^\prime(b)+2b(1\!-\!b)f^{\prime\prime}(b)=:g(b)\pagebreak$$ 
  as $s\!\rightarrow\! 0+$. For $b\!=\!0$, the same argument applies if $r\!>\!0$. If $r\!\le\! 0$ absorption yields a zero limit, which extends $g$ continuously if and only if $f^\prime(0)\!=\!0$ or $r\!=\!0$. The analogous argument at $b\!=\!1$ requires $r^\prime\!\ge\! 0$ or $f^\prime(1)\!=\!0$.
\end{proof}

\subsection{Wright--Fisher processes ${\tt WF}_{\mathbf{b}}(\mathbf{r})$}

Consider parameters $\ell\ge 2$, $\mathbf{r}=(r_1,\ldots,r_\ell)\!\in\!\mathbb{R}^\ell$ and initial state $\mathbf{b}=(b_1,\ldots,b_\ell)\in\Delta_\ell:=\big\{(x_1,\ldots,x_\ell)\!\in\![0,1]^\ell\colon\sum_{i\in[\ell]}x_i\!=\!1\big\}$, where we wrote $[\ell]\!:=\!\{1,\ldots,\ell\}$. Set $r_+\!:=\!\sum_{i\in[\ell]}r_i$. Pal \cite{Pal11,Pal13} adapted the Warren--Yor construction of Jacobi diffusions to construct diffusions on the simplex $\Delta_\ell$. Specifically, consider independent $Z^{(i)}\sim{\tt BESQ}_{b_i}(2r_i)$, $i\!\in\![\ell]$, and denote by 
  $$S_0=\inf\big\{s\!\ge\! 0\colon Z^{(i)}_s\!=\!0\mbox{ for some }i\!\in\![\ell]\mbox{ with }r_i\!<\!0\big\}$$ 
the first absorption time of a ${\tt BESQ}$ with negative parameter. On $[0,S_0)$, consider $Z^{(+)}\!:=\!\sum_{i\in\ell}Z^{(i)}$ and the time-change\vspace{-0.1cm}
\begin{equation}\label{eq:timechangeWF}\rho(t)=\inf\bigg\{s\!\ge\! 0\colon\!\int_0^s\!\frac{dv}{Z^{(+)}_v}>t\bigg\},\quad 0\!\le\! t\!<\!T:=\int_0^{S_0}\!\frac{dv}{Z^{(+)}_v}.
\end{equation}
Then $\Big(\big(Z^{(+)}_{\rho(t\wedge T)}\big)^{-1}Z^{(1)}_{\rho(t\wedge T)},\ldots,\big(Z^{(+)}_{\rho(t\wedge T)}\big)^{-1}Z^{(\ell)}_{\rho(t\wedge T)}\Big)$, $t\ge 0$, the stopped and time-changed proportions of $(Z^{(1)},\ldots,Z^{(\ell)})$, form a $\Delta_\ell$-valued diffusion, whose distribution we denote by ${\tt WF}_{\mathbf{b}}(\mathbf{r})$. When
$r_1,\ldots,r_\ell\ge 0$, this is (up to a linear time-change) the well-known Wright--Fisher diffusion, see e.g. \cite{EthKurtzBook}. In particular,
$\mathbf{W}=\big(W^{(1)},\ldots,W^{(\ell)}\big)\sim{\tt WF}_{\mathbf{b}}(\mathbf{r})$ satisfies the SDEs
$$dW^{(i)}_t=2\left(1-W^{(i)}_t\right)\sqrt{W^{(i)}_t}dB_t^{(i)}-2W^{(i)}_t\!\sum_{j\in[\ell]\setminus\{i\}}\!\sqrt{W^{(j)}_t}dB_t^{(j)}\,+\,2\left(r_i-r_+W^{(i)}_t\right)dt,$$
with $W^{(i)}_0=b_i$, for all $i\in[\ell]$, where $(B^{(1)},\ldots,B^{(\ell)})$ is a vector of independent Brownian motions. Also, 
${\tt WF}(\mathbf{r})$ has infinitesimal generator
\begin{equation}\label{eq:WFgen}\mathcal{A}_{\tt WF}^{\mathbf{r}}=2\sum_{i\in[\ell]}w_i\frac{\partial^2}{\partial w_i^2}-2\sum_{i,j\in[\ell]}w_iw_j\frac{\partial^2}{\partial w_i\partial w_j}-2\sum_{i\in[\ell]}\big(r_+w_i-r_i\big)\frac{\partial}{\partial w_i}.
\end{equation}
The extension to negative parameters was observed by Pal \cite{Pal13}. The arguments are also valid for $\mathbf{r}$ with both negative and nonnegative entries. 
\begin{lemma}\label{lm:domWF} The domain of $\mathcal{A}^{\mathbf{r}}_{\tt WF}$ includes all functions $f\colon\Delta_\ell\rightarrow\mathbb{R}$ that possess an extension to $\mathbb{R}^\ell$ that is twice continuously differentiable and further satisfies
  $$\frac{\partial}{\partial w_i}f(\mathbf{w})=0\quad\mbox{for all $\mathbf{w}\in\Delta_\ell$ with $w_i=0$, if $r_i<0$}\quad\mbox{for all $i\in[\ell]$.}$$
\end{lemma}
The proof of Lemma \ref{lm:domJAC} is easily adapted to this $\ell$-dimensional setting.

\subsection{Properties of ${\tt SSSP}(\alpha,\theta)$ and ${\tt FV}(\alpha,\theta)$}

The reader will have observed the parallels between the de-Poissonization time-change constructions of ${\tt FV}(\alpha,\theta)$ from 
${\tt SSSP}(\alpha,\theta)$ and of ${\tt JAC}(r,r^\prime)$ and ${\tt WF}(\mathbf{r})$ from vectors of ${\tt BESQ}$ processes. 
Let us here recall from \cite{PaperFV} some properties of ${\tt SSSP}(\alpha,\theta)$ and ${\tt FV}(\alpha,\theta)$ that shed more light on 
these parallels. 

For ${\tt FV}(\alpha,\theta)$, the time-change $t\mapsto\rho(t)$ only depends on $\big(\|\mu_s\|,s\!\ge\! 0\big)$, the total mass process of the
${\tt SSSP}(\alpha,\theta)$. For ${\tt JAC}(r,r^\prime)$ and ${\tt WF}(\mathbf{r})$, the corresponding quantity is the sum of all independent 
${\tt BESQ}$ processes, which is a ${\tt BESQ}(2r+2r^\prime)$ or ${\tt BESQ}(2r_+)$ by the well-known \cite{ShigWata73} additivity of 
${\tt BESQ}$ when all parameters are nonnegative or natural extensions (subject to suitable stopping) when some parameters are negative, as noted previously by the present authors \cite{Paper1-1}, see also \cite{PW18}. 

\begin{proposition}[Theorem 1.5 of \cite{PaperFV}]\label{prop:totalmass} For $\big(\mu_s,\,s\ge 0\big)\sim{\tt SSSP}_\mu(\alpha,\theta)$, we have $\big(\|\mu_s\|,\,s\ge 0\big)\sim{\tt BESQ}(2\theta)$.
\end{proposition}

By definition, this total mass process is the sum of countably many atom sizes at all times, but the additivity of ${\tt BESQ}$ enters in a more subtle way. The transition semi-group $(K_s^{\alpha,\theta}\!,s\!\ge \!0)$\linebreak of ${\tt SSSP}(\alpha,\theta)$ stated in Definition \ref{def:kernel:sp} leaves 
implicit the evolution of atoms and sheds little light on the creation of new atoms. In \cite{PaperFV}, we provide a Poissonian construction that explicitly specifies independent ${\tt BESQ}(-2\alpha)$ evolutions for each atom size and creates new atoms at times corresponding to pre-jump levels in a ${\tt Stable}(1+\alpha)$ L\'evy process. We do not need the details of this construction in the present paper and refer the reader to 
\cite{PaperFV}, but the following consequence of the Poissonian construction is important for us.

\begin{proposition}[Corollary 5.11 of \cite{PaperFV}]\label{prop:emimm} Let $\big(\cev{\mu}_s,s\!\ge\! 0\big)\sim{\tt SSSP}_0(\alpha,\alpha)$ and, 
  independently, consider $Z\sim{\tt BESQ}_b(-2\alpha)$ with absorption time $S$. Set $\mu_s=Z_s\delta(u)+\cev{\mu}_s$, $0\le s\le S$. 
  Conditionally given $\big(\mu_s,\,0\le s\le S\big)$ with $\mu_S=\lambda$, let $\big(\mu_{S+v},\,v\!\ge\! 0\big)\sim{\tt SSSP}_\lambda(\alpha,0)$. Then 
  $\big(\mu_s,\,s\!\ge\! 0\big)\sim{\tt SSSP}_{b\delta(u)}(\alpha,0)$.
\end{proposition}

Recall from the introduction that the transition kernels $K_s^{\alpha,\theta}$, $s\ge 0$, of ${\tt SSSP}(\alpha,\theta)$ stated in Definition \ref{def:kernel:sp} possess a branching property that suggests an ancestral relationship between any time-$0$ atom $b_i\delta(u_i)$ and 
the time-$s$ atoms of $\Pi_i$, for each $i\ge 1$. We can interpret the remaining time-$s$ atoms of $G_0\overline{\Pi}_0$ as \em immigration\em.  
The following result expresses this split at a fixed time, via the Markov property, in terms of independent superprocesses. 

\begin{proposition}[Proposition 1.4 and Theorem 1.10 of \cite{PaperFV}]\label{prop:split} For any finite measure 
  $\mu=\sum_{i\ge 1}b_i\delta(u_i)$, consider independent $\mu^{(0)}\sim{\tt SSSP}_0(\alpha,\theta)$ and 
  $\mu^{(i)}\sim{\tt SSSP}_{b_i\delta(u_i)}(\alpha,0)$. Then $(\mu_s,\,s\ge 0):=\sum_{i\ge 0}\mu^{(i)}\sim{\tt SSSP}_{\mu}(\alpha,\theta)$.
\end{proposition}

\subsection{Jacobi and Wright--Fisher processes associated with ${\tt FV}(\alpha,\theta)$}

The following result records the consequences for ${\tt FV}(\alpha,\theta)$ of the ${\tt BESQ}$ processes associated with 
${\tt SSSP}(\alpha,\theta)$ by combining Propositions \ref{prop:totalmass}--\ref{prop:split}. 

\begin{proposition}\label{prop:wfinfv} In the setting of Proposition \ref{prop:split}, with an initial probability measure $\mu\!=\!\sum_{i\ge 1}p_i\delta(u_i)$, denote by $M_t:=\big(\|\mu_t\|,\,t\ge 0\big)$ the total mass process and by $(\rho(t),\,t\ge 0)$ the time-change of \eqref{eq:timechangeFV}. Let $k\ge 1$. 
  \begin{enumerate} \item[(i)] Then $X^{(k)}\!:=\!\big((M_{\rho(t)})^{-1}\|\mu_{\rho(t)}^{(k)}\|,t\!\ge\! 0\big)\!\sim\!{\tt JAC}_{p_k}(0,\theta)$, $X^{(0)}\!:=\!\big((M_{\rho(t)})^{-1}\|\mu_{\rho(t)}^{(0)}\|,t\!\ge\! 0\big)$ $\sim\!{\tt JAC}_{0}(\theta,0)$, and 
  $$\Big(X^{(1)},\ldots,X^{(k)},1\!-\!\sum\nolimits_{i\in[k]}X^{(i)}\Big)\sim{\tt WF}_{(p_1,\ldots,p_k,1\!-\!\sum_{i\in[k]}p_i)}(0,\ldots,0,\theta).$$
\item[(ii)] We also have $W^{(k)}\!:=\!\big((M_{\rho(t)})^{-1}\mu^{(k)}_{\rho(t)}\!\{u_i\},t\!\ge\! 0\big)\!\sim\!{\tt JAC}_{p_k}(-\alpha,\theta\!+\!\alpha)$, and for
 $$W^{(-k)}\!:=\!\big((M_{\rho(t\wedge T_k)})^{-1}\mu^{(k)}_{\rho(t\wedge T_k)}\big([0,1]\setminus\{u_k\}\big),t\!\ge\! 0\big),$$
 where $T_k$ is the absorption time of $W^{(k)}\!$, we also have 
 $$\big(W^{(k)},W^{(-k)},1\!-\!W^{(k)}\!-\!W^{(-k)}\big) \sim {\tt WF}_{(p_k,0,1-p_k)}(-\alpha,\alpha,\theta).$$ 
 Furthermore, for $T=\min\{T_1,\ldots,T_k\}$, we have
 $$\Big(\Big(W^{(1)}_{t\wedge T},\ldots,W^{(k)}_{t\wedge T},1\!-\!\sum\nolimits_{i\in[k]}\!W^{(i)}_{t\wedge T}\Big),\,t\ge 0\Big) \sim {\tt WF}_{\mathbf{b}}(\mathbf{r})$$
 with $\mathbf{r}\!=\!(-\alpha,\ldots,-\alpha,\theta\!+\!k\alpha)$ and $\mathbf{b}=(p_1,\ldots,p_k,1-\sum_{i\in[k]}p_k)$.
  \end{enumerate}
\end{proposition}
\begin{proof} (i) Proposition \ref{prop:totalmass} applied to each $\mu^{(i)}$, $i\ge 0$, yields independent 
  $M^{(0)}:=\big(\|\mu_s^{(0)}\|,\,s\ge 0\big)\sim{\tt BESQ}_0(2\theta)$ and $M^{(k)}:=\big(\|\mu_s^{(k)}\|,\,s\ge 0\big)\sim{\tt BESQ}_{p_k}(0)$, 
  $k\ge 1$. By the additivity of ${\tt BESQ}$, we further note that 
  $N^{(k)}:=\sum_{i\in\mathbb{N}_0\setminus\{k\}}M^{(i)}\sim{\tt BESQ}_{1-p_k}(2\theta)$, and that $M^{(k)}$ and $N^{(k)}$ are independent, for each $k\ge 1$. Similarly, $R^{(k)}:=\sum_{i\in\mathbb{N}_0\setminus[k]}M^{(i)}\sim{\tt BESQ}_{1-\sum_{i\in[k]}p_i}(2\theta)$ 
  is independent of $\big(M^{(1)},\ldots,M^{(k)}\big)$. 
  
  The time-change $(\rho(t),\,t\ge 0)$ of \eqref{eq:timechangeFV} is based on the total mass process
  $M\!:=\!\big(\|\mu_s\|,s\!\ge\!0\big)$. But since $M\!=\!M^{(k)}\!+\!N^{(k)}\!=\!R^{(k)}\!+\!\sum_{i\in[k]}M^{(i)}$ for all $k\!\ge\! 1$, this is the same time-change as
  \eqref{eq:timechangeJAC} to construct $X^{(k)}\!\sim\!{\tt JAC}_{p_k}(0,\theta)$ from $Z\!:=\!M^{(k)}$ and $Z^\prime\!:=\!N^{(k)}$. This
  time-change is also the same as \eqref{eq:timechangeWF} to construct ${\tt WF}_{\mathbf{b}}(\mathbf{r})$ with $\mathbf{b}=(p_1,\ldots,p_k,1\!-\!\sum_{i\in[k]}p_i)$ and $\mathbf{r}=(0,\ldots,0,\theta)$,  $M^{(i)}$ as $Z^{(i)}$, $i\in[k]$, and $R^{(k)}$ as $Z^{(k+1)}$. In particular, the first $k$ components of the ${\tt WF}_{\mathbf{b}}(\mathbf{r})$ are indeed 
  $\big(X^{(1)},\ldots,X^{(k)}\big)$, and the last component is as required to add to 1.
  
  (ii) We refine the setting of (i). If we furthermore construct each $\mu^{(i)}$, $i\ge 1$, as in Proposition \ref{prop:emimm}, we instead 
  obtain a countable family of independent $Z^{(i)}:=\big(\mu_s^{(i)}\{u_i\},\,s\ge 0\big)\sim{\tt BESQ}_{p_i}(-2\alpha)$, $i\ge 1$. Now applying Proposition \ref{prop:totalmass} to 
  $\mu^{(0)}$ and $\cev{\mu}^{(i)}$, we also have independent $Z^{(0)}:=\big(\|\mu_s^{(0)}\|,\,s\ge 0\big)\sim{\tt BESQ}_0(2\theta)$ and 
  $Z^{(-i)}:=\big(\|\cev{\mu}^{(i)}_s\|,\,s\ge 0\big)\sim{\tt BESQ}_0(2\alpha)$, $i\ge 1$. 
  
  Recall notation $N^{(j)}$ from the proof of (i). Note that the independence of $\mu^{(j)}$ and 
  $\sum_{i\in\mathbb{N}_0\setminus\{j\}}\mu^{(i)}$ entails that $Z^{(j)}\!\sim\!{\tt BESQ}_{p_j}(-2\alpha)$ is independent of 
  $Z^{(-j)}\!\sim\!{\tt BESQ}_0(2\alpha)$ and $N^{(j)}\!\sim\!{\tt BESQ}_{1-p_j}(2\theta)$, and hence independent of their sum 
  $L^{(j)}\!:=\!Z^{(-j)}\!+\!N^{(j)}\sim{\tt BESQ}_{1-p_j}(2(\theta\!+\!\alpha))$, as required to get
  $W^{(j)}\!:=\!\big((M_{\rho(t)})^{-1}Z^{(j)}_{\rho(t)},t\!\ge\! 0)\!\sim\!{\tt JAC}_{p_j}(-\alpha,\theta\!+\!\alpha)$, and indeed as required to construct
  $(W^{(j)}\!,W^{(-j)}\!,1\!-\!W^{(j)}\!-\!W^{(-j)})\!\sim\!{\tt WF}_{(p_j,0,1-p_j)}(-\alpha,\alpha,\theta)$, where we recall that this process is stopped at 
  the time that the left-most component hits 0.
  
  Assembling several ${\tt JAC}_{p_i}(-\alpha,\theta+\alpha)$ to ${\tt WF}_{\mathbf{b}}(-\alpha,\ldots,-\alpha,\theta+k\alpha)$ can be done as in (i), with the caveat that having $k$ negative parameters makes this construction (and the definition of ${\tt WF}$) only valid/useful up to the random time 
  $$\!\!\!\!T=\inf\big\{t\ge 0\colon\exists_{i\in[k]}\,\pi_t\{u_i\}=0\big\}=\inf\big\{t\ge 0\colon\exists_{i\in[k]}\,W^{(i)}_t=0\}=\min\{T_1,\ldots,T_k\}.\vspace{-0.5cm}$$
\end{proof}

\begin{corollary}\label{cor:katoms} Consider $(\pi_t,t\!\ge\!0)\!\sim\!{\tt FV}_{\mu}(\alpha,\theta)$ starting from any probability measure $\mu\!=\!\sum_{i\ge 1}p_i\delta(u_i)$, then for any $j\ge 1$, we have $(\pi_{t}\{u_j\},t\!\ge\! 0)\!\sim\!{\tt JAC}_{p_j}(-\alpha,\theta\!+\!\alpha)$. Also, 
   $$\big((\pi_{t\wedge T}\{u_1\},\ldots,\pi_{t\wedge T}\{u_k\},\pi_{t\wedge T}([0,1]\!\setminus\!\{u_1,\ldots,u_k\}),\,t\!\ge\! 0\big)\sim{\tt WF}_{\mathbf{b}}(\mathbf{r}),$$ for any $k\ge 1$, where $T\!=\!\inf\{t\!\ge\! 0\colon\!\exists_{i\in[k]}\,\pi_t\{u_i\}\!=\!0\}$, $\mathbf{b}\!=\!(p_1,\ldots,p_k,1\!-\!\sum_{i\in[k]}p_i)$ and $\mathbf{r}\!=\!(-\alpha,\ldots,-\alpha,\theta+k\alpha)$.
\end{corollary}
\begin{proof} In the setting of Proposition \ref{prop:wfinfv}(ii), we have $(\pi_t,\,t\ge 0):=\big(\|\mu_{\rho(t)}\|^{-1}\mu_{\rho(t)}\big)\sim{\tt FV}_\mu(\alpha,\theta)$, and 
  $W^{(j)}=\big(\pi_t\{u_j\},\,t\ge 0\big)$ a.s., so $(\pi_t\{u_j\},\,t\ge 0)\sim{\tt JAC}_{p_j}(-\alpha,\theta+\alpha)$. This also entails the
  ${\tt WF}_{\mathbf{b}}(\mathbf{r})$ claim.
\end{proof}

\section{Generators and semi-groups on $\mathbf{L}^2[\alpha,\theta]$}\label{sec:Petrov}

Recall that Theorem \ref{thm:petroviden} claims that for $(\pi_t,t\!\ge\! 0)\sim{\tt FV}_\pi(\alpha,\theta)$, the projection 
$\big({\tt RANKED}(\pi_t),t\!\ge\! 0\big)$ is ${\tt EKP}_{{\tt RANKED}(\pi)}(\alpha,\theta)$. 
The aim of this section is to identify the $\mathbf{L}^2[\alpha,\theta]$-semi-group of the projected process. Specifically, we establish the Markov property in Section \ref{sec:Markov}. In Section \ref{sec:generator}, we compute the infinitesimal generator on Petrov's algebra $\mathcal{F}$, and in Section \ref{sec:identification} we conclude that the $\mathbf{L}^2[\alpha,\theta]$-semi-groups of the projected process and of ${\tt EKP}_{{\tt RANKED}(\pi)}(\alpha,\theta)$ coincide.

\subsection{Markov property of $\big({\tt RANKED}(\pi_t),\,t\ge 0\big)$}\label{sec:Markov}

To study the projection of ${\tt FV}(\alpha,\theta)$ to $\nabla_{\!\infty}$, let us introduce notation $\mathcal{M}^a_1$ for the set of all purely atomic probability measures on the Borel sigma-algebra $\mathcal{B}([0,1])$ of the interval $[0,1]$. We consider two topologies on $\mathcal{M}^a_1$, the weak topology, which is separable, and the topology induced by the total variation distance
$$d_{\tt TV}(\pi,\pi^\prime)=\sup_{B\in\mathcal{B}([0,1])}\left|\pi(B)-\pi^\prime(B)\right|,$$
which is not separable. We will denote by $\mathbb{P}^{\alpha,\theta}_\pi$ a probability measure under which 
$(\pi_t,t\!\ge\! 0)$ $\sim{\tt FV}_\pi(\alpha,\theta)$, and by $\mathbb{E}^{\alpha,\theta}_\pi$ associated expectations.

Let us also clarify the topology on the Kingman simplex
\begin{equation}
 \nabla_{\!\infty} := \bigg\{\mathbf{x}=\left(x_1, x_2, \ldots   \right)\colon x_1 \ge x_2 \ge \cdots \ge 0,\; \sum_{i\ge 1} x_i = 1     \bigg\}.
\end{equation}
This is a metric space under $\ell^{\infty}$. Its closure under $\ell^{\infty}$, denoted by $\nabla_{\!\infty}cl$, is the set of non-increasing sequences in $[0,1]$ with sum at most 1. Petrov \cite{Petrov09} established ${\tt EKP}(\alpha,\theta)$ as path-continuous Markov processes that can start anywhere in $\overline{\nabla}_{\!\infty}$.  
It has been shown in \cite{Eth14} that the processes, starting at $\mathbf{x}\in\nabla_{\!\infty}$, never leave $\nabla_{\!\infty}$. Therefore, it is already known that ${\tt EKP}(\alpha,\theta)$ can be considered as diffusions on $\nabla_{\!\infty}$. Since the ${\tt RANKED}$ map takes values in $\nabla_{\!\infty}$ only and is surjective onto $\nabla_{\infty}$, this also follows from our Theorem \ref{thm:petroviden}.

\begin{lemma}\label{lm:ranked} ${\tt RANKED}\colon\mathcal{M}^a_1\rightarrow\nabla_{\!\infty}$ is Borel measurable with respect to the weak topology and is $d_{\tt TV}$-continuous.
\end{lemma}
\begin{proof} It is well-known -- see e.g. \cite[Lemma 1.6]{Kallenberg2017} -- that there are measurable enumeration maps that associate with $\pi\in\mathcal{M}_1^a$ the countable sequence of all location/size pairs of atoms, which can then be ranked measurably. Furthermore, ${\tt RANKED}$ is Lipschitz with respect
  to $d_{\rm TV}$ and $\ell^\infty$.
\end{proof}

\begin{proposition}\label{prop:Markov} Let $\alpha\in(0,1)$, $\theta\ge 0$, $\pi\in\mathcal{M}^a_1$ and $(\pi_t,t\!\ge\! 0)\sim{\tt FV}_\pi(\alpha,\theta)$. 
  Then $\big({\tt RANKED}(\pi_t),t\!\ge\! 0\big)$ is a path-continuous $\nabla_{\!\infty}$-valued Markov process that is
  stationary with respect to the ${\tt PD}(\alpha,\theta)$ law.
\end{proposition}
\begin{proof} We will show that for any two $\pi^\prime,\pi^{\prime\prime}\in\mathcal{M}^a_1$ with ${\tt RANKED}(\pi^\prime)={\tt RANKED}(\pi^{\prime\prime})$, we
  can couple ${\tt FV}_{\pi^\prime}(\alpha,\theta)$ and ${\tt FV}_{\pi^{\prime\prime}}(\alpha,\theta)$ processes that have the same projection under ${\tt RANKED}$. 
  
  First consider $\mu^\prime\!=\!b\delta(u^\prime)$ and $\mu^{\prime\prime}\!=\!b\delta(u^{\prime\prime})$. By Proposition \ref{prop:emimm}, we can construct 
  ${\tt SSSP}_{b\delta(u)}(\alpha,0)$ from independent $\cev{\mu}\!\sim\!{\tt SSSP}_0(\alpha,\alpha)$ and $Z\!\sim\!{\tt BESQ}_{b}(-2\alpha)$, with
  absorption time $S=\inf\{s\ge 0\colon Z_s=0\}$, and from $(\widetilde{\mu}_v,\,v\ge 0)$, an ${\tt SSSP}(\alpha,0)$ starting from 
  $\cev{\mu}_{S}$. None of this depends on $u$ and we can set
  $$\mu_s^\prime:=\left\{\begin{array}{ll}Z_s\delta(u^\prime)+\cev{\mu}_s,&0\le s<S,\\
                                                        \widetilde{\mu}_{s-S},&s\ge S,
                                  \end{array}\right.
      \quad\mu_s^{\prime\prime}:=\left\{\begin{array}{ll}Z_s\delta(u^{\prime\prime})+\cev{\mu}_s,&0\le s<S,\\
                                                        \widetilde{\mu}_{s-S},&s\ge S,
                                  \end{array}\right.                            $$
  to couple $(\mu_s^\prime,\,s\ge 0)\sim{\tt SSSP}_{b\delta(u^\prime)}(\alpha,0)$ and $(\mu_s^{\prime\prime},\,s\ge 0)\sim{\tt SSSP}_{b\delta(u^{\prime\prime})}(\alpha,0)$. 
  
  In general, we can write $\pi^\prime=\sum_{i\ge 1}b_i\delta(u_i^\prime)$ and $\pi^{\prime\prime}=\sum_{i\ge 1}b_i\delta(u_i^{\prime\prime})$ for the same sequence 
  $(b_i,\,i\ge 1)$. We construct an ${\tt SSSP}_{\pi^\prime}(\alpha,\theta)$ and an ${\tt SSSP}_{\pi^{\prime\prime}}(\alpha,\theta)$ as in 
  Proposition \ref{prop:split}, with each pair of ${\tt SSSP}_{b_i\delta(u_i^\prime)}(\alpha,0)$ and ${\tt SSSP}_{b_i\delta(u_i^{\prime\prime})}$, $i\ge 1$, 
  coupled as above, and using the same ${\tt SSSP}_0(\alpha,\theta)$. 
  
  This coupling is such that all ranked masses at all times coincide, for the ${\tt SSSP}_{\pi^\prime}(\alpha,\theta)$ and the 
  ${\tt SSSP}_{\pi^{\prime\prime}}(\alpha,\theta)$. In particular, they have the same total mass processes and the same time-change 
  \eqref{eq:timechangeFV}, and therefore the associated ${\tt FV}_{\pi^\prime}(\alpha,\theta)$ and ${\tt FV}_{\pi^{\prime\prime}}(\alpha,\theta)$ share the
  same ranked mass processes. 
  
  Let $F\colon \nabla_{\!\infty} \to [0,\infty)$ be bounded measurable. For the coupled processes $(\pi_t^\prime,\,t\ge 0)\!\sim\!{\tt FV}_{\pi^\prime}(\alpha,\theta)$ and
  $(\pi_t^{\prime\prime},\,t\ge 0)\sim{\tt FV}_{\pi^{\prime\prime}}(\alpha,\theta)$, we have $\mathbb{E}\big[F({\tt RANKED}(\pi_t^\prime)) \big]=\mathbb{E}\big[F({\tt RANKED}(\pi_t^{\prime\prime})) \big]$. In particular, $\mathbb{E}_\pi^{\alpha,\theta}\big[F({\tt RANKED}(\pi_t))\big]$ is a function of ${\tt RANKED}(\pi)\in\nabla_{\!\infty}$. By Dynkin's criterion (e.g. \cite[Lemma I.14.1]{RogersWilliams}), mapping ${\tt FV}(\alpha,\theta)$ via ${\tt RANKED}$ yields a Markov process. By \cite[Corollary 5.5]{PaperFV}, 
  ${\tt SSSP}(\alpha,\theta)$ and hence ${\tt FV}(\alpha,\theta)$ are $d_{\tt TV}$-path-continuous. Since
  ${\tt RANKED}$ is $d_{\tt TV}$-continuous, by Lemma \ref{lm:ranked}, mapping ${\tt FV}(\alpha,\theta)$ under ${\tt RANKED}$ yields a path-continuous process in $\nabla_{\!\infty}$. Mapping a stationary ${\tt FV}(\alpha,\theta)$, with ${\tt PDRM}(\alpha,\theta)$ stationary distribution clearly yields a process that has stationary distribution ${\tt PD}(\alpha,\theta)$.
\end{proof}

The same method allows us to prove that the projected process is a Hunt process, which will also follow from our identification with the $\EKP$ diffusion.

\subsection{The infinitesimal generator of $({\tt RANKED}(\pi_t),\,t\ge 0)$ on the algebra $\mathcal{F}$}\label{sec:generator}

We will frequently employ an (arbitrary) inclusion map $\iota\colon \nabla_{\!\infty}\rightarrow\mathcal{M}^a_1$:
\[
\iota(\mathbf{x}) = \sum_{i\ge 1}x_i\delta(u_i),\qquad\mbox{where }u_i=1/i,\ i\ge 1.
\]
We will abuse notation and write ${\tt FV}_{\mathbf{x}}(\alpha,\theta):={\tt FV}_{\iota(\mathbf{x})}(\alpha,\theta)$ and $\mathbb{E}_{\mathbf{x}}^{\alpha,\theta}:=\mathbb{E}_{\iota(\mathbf{x})}^{\alpha,\theta}$. We will also follow the convention of including finite-dimensional unit simplices in $\nabla_{\!\infty}$ by appending zeros.

\begin{proposition}\label{lem:genidentify} For every $q\in\mathcal{F}$ we have 
\begin{equation}\label{eq:genvisa}
\lim_{t\rightarrow 0+} \frac{\mathbb{E}^{\alpha,\theta}_{\mathbf{x}}\left[q\left({\tt RANKED}(\pi_t)\right) \right] - q(\mathbf{x}) }{t} =2 \petrovgen q(\mathbf{x}), \quad \text{for every $\mathbf{x} \in \nabla_{\!\infty}$},
\end{equation}
where $\petrovgen$ is (the restriction to $\mathcal{F}$ of) the generator \eqref{petrovgen} of $\EKP[\alpha,\theta]$.
The above convergence also holds in $\ltwo$ with respect to the law of $\PoiDir[\alpha,\theta]$. 
\end{proposition}

Proposition \ref{lem:genidentify} is proved in two steps: first we prove \eqref{eq:genvisa} when $q=q_m$ for some $m\ge 1$, and then for the general case. Recall
that $q_m(\mathbf{x})=\sum_{i\ge 1}x_i^{m+1}$. In general, $q_m({\tt RANKED}(\pi_t))$ is a sum over many atoms, cf.\ Definition \ref{def:kernel:sp} for the transition kernel before the de-Poissonization time-change/normalization. We will work with lower and upper bounds on  $\mathbb{E}^{\alpha,\theta}_{\mathbf{x}}\left[q\left({\tt RANKED}(\pi_t)\right) \right] - q(\mathbf{x})$ 
that separate the main contributions and asymptotically negligible contributions. 

To prepare this, we first establish three lemmas. In these lemmas we use the setting of Propositions \ref{prop:split} and \ref{prop:wfinfv}, with
$(\pi_t,t\!\ge\! 0)\!\sim\!{\tt FV}_{\mathbf{x}}(\alpha,\theta)$ constructed from independent $\mu^{(i)}$, $i\!\ge\! 0$, which are one ${\tt SSSP}(\alpha,0)$ starting from from each initial atom and one more ${\tt SSSP}_0(\alpha,\theta)$, ``immigration''. We denote by $M_t:=\sum_{i\ge 0}\big\|\mu_t^{(i)}\big\|$ the total mass process so that
\begin{equation}\label{eq:splitpit}\pi_t=\sum_{i\ge 0}\pi_t^{(i)},\quad
    \mbox{where }\ \pi_t^{(i)}:=\frac{\mu_{\rho(t)}^{(i)}}{M_{\rho(t)}}\ \mbox{ and }\ \rho(t)=\inf\bigg\{s\ge 0\colon\int_0^s\frac{dv}{M_v}>t\bigg\}.
\end{equation}
In particular, this gives access to two Jacobi processes for each $i\ge 1$:
$$W^{(i)}_t:=\pi_t\{u_i\}=\big(M_{\rho(t)}\big)^{-1}\mu^{(i)}_{\rho(t)}\{u_i\}\le\big(M_{\rho(t)}\big)^{-1}\big\|\mu^{(i)}_{\rho(t)}\big\|=:X^{(i)}_t.$$

\begin{lemma}\label{lm:lower:main} For each $i\ge 1$, we have, as $t\rightarrow 0+$,
  \begin{align*}&\frac{1}{t}\Big(\mathbf{E}\Big[\big(W^{(i)}_t\big)^{m+1}\Big]-x_i^{m+1}\Big)\rightarrow 2(m\!+\!1)(m\!-\!\alpha)x_i^m-2(m\!+\!1)(m\!+\!\theta)x_i^{m+1}\\
  \text{and }&\displaystyle\frac{1}{t}\Big(\mathbf{E}\Big[\big(X^{(i)}_t\big)^{m+1}\Big]-x_i^{m+1}\Big)\rightarrow 2(m\!+\!1)mx_i^m-2(m\!+\!1)(m\!+\!\theta)x_i^{m+1}.
  \end{align*}
\end{lemma}
\begin{proof} By Proposition \ref{prop:wfinfv}, $W^{(i)}\!\!\sim\!{\tt JAC}_{x_i}(-\alpha,\theta+\alpha)$ 
and $X^{(i)}\!\!\sim\!{\tt JAC}_{x_i}(0,\theta)$. 
By Lemma \ref{lm:domJAC}, we may apply the generator \eqref{eq:JACgen} to $f(x)=x^{m+1}$.
\end{proof}
Since $2\mathcal{B}q_m(\mathbf{x})=2(m+1)(m-\alpha)\sum_{i\ge 1}x_i^m-2(m+1)(m+\theta)\sum_{i\ge 1}x_i^{m+1}$, the first of these captures the main contributions. We also need some uniform bounds on these quantities.

\begin{lemma}\label{lm:upper:main} For each $i\ge 1$ and all $m\ge 1$, $\mathbf{x}\in\nabla_{\!\infty}$ and $t\ge 0$, we have
  \begin{align*}
    -2(m\!+\!1)(m\!+\!\theta)x_i&\le\frac{1}{t}\Big(\mathbf{E}\Big[\big(W^{(i)}_t\big)^{m+1}\Big]-x_i^{m+1}\Big)\\
     &\le\frac{1}{t}\Big(\mathbf{E}\Big[\big(X^{(i)}_t\big)^{m+1}\Big]-x_i^{m+1}\Big)\le 2(m\!+\!1)m x_i.
  \end{align*}
\end{lemma}
\begin{proof} By Proposition \ref{prop:wfinfv}, $X^{(i)}\!\!\sim\!{\tt JAC}_{x_i}(0,\theta)$. By It\^o's formula, we have
  $\mathbf{E}\big[\big(X_v^{(i)}\big)^m\big]\le\mathbf{E}\big[X_v^{(i)}\big]\le x_i$, so applying It\^o's formula again yields 
  $$\frac{1}{t}\Big(\mathbf{E}\Big[\big(X^{(i)}_t\big)^{m+1}\Big]-x_i^{m+1}\Big)\le\frac{1}{t}\mathbf{E}\left[\int_0^t2(m+1)m\big(X_v^{(i)}\big)^mdv\right]\le 2(m+1)mx_i.$$
  Similarly, we bound $t^{-1}\big(\mathbf{E}\big[\big(W^{(i)}_t\big)^{m+1}\big]-x_i^{m+1}\big)$ below by
  $$-\frac{1}{t}\mathbf{E}\left[\int_0^t2(m+1)(m+\theta)\big(W_v^{(i)}\big)^{m+1}dv\right]\ge-2(m+1)(m+\theta)x_i.
  \vspace{-0.6cm}$$
\end{proof}

Finally, we control some asymptotically negligible contributions, using notation $X_t^{(0)}\!=\!\big(M_{\rho(t)}\big)^{-1}\big\|\mu_{\rho(t)}^{(0)}\big\|$ and $W_t^{(-i)}\!=\!X_t^{(i)}\!\!-\!W_t^{(i)}$ of Proposition \ref{prop:wfinfv}.

\begin{lemma}\label{lm:upper:negl} We have $t^{-1}\mathbf{E}\Big[\big(X^{(0)}_t\big)^{m+1}\Big]\!\rightarrow\! 0$ and
  $t^{-1}\mathbf{E}\Big[\big(W^{(-i)}_t\big)^{m+1}\Big]\!\rightarrow\! 0$ as $t\!\rightarrow\!0+$, and
  $t^{-1}\mathbf{E}\Big[\big(W^{(-i)}_t\big)^{m+1}\Big]\le 4(2+\theta)x_i$ for all $t>0$, for all $i\ge 1$.
\end{lemma}
\begin{proof} By Proposition \ref{prop:wfinfv}, $X^{(0)}\!\!\sim\!{\tt JAC}_{0}(\theta,0)$, so we can apply the generator to $f(x)=x^{m+1}$
  and evaluate at $x=0$. Since $W^{(-i)}$ is not itself a ${\tt JAC}_0(\alpha,\theta-\alpha)$ as it has been stopped when $W^{(i)}$ vanishes, we consider $(W^{(i)},W^{(-i)},1-W^{(i)}-W^{(-i)})\sim{\tt WF}_{(x_i,0,1-x_i)}(-\alpha,\alpha,\theta)$. By Lemma 
  \ref{lm:domWF}, we can apply the generator \eqref{eq:WFgen} to $f(\mathbf{w})=w_2^{m+1}$ and evaluate at $w_2=0$.
  
  Finally, $\big(W_t^{(-i)}\big)^{m+1}\le\big(X_t^{(i)}-W_t^{(i)}\big)^2\le \big(\big(X_t^{(i)}\big)^2-x_i^2\big)-\big(\big(W_t^{(i)}\big)^2-x_i^2\big)$, so the bound follows by taking $m=1$ in Lemma \ref{lm:upper:main}. 
\end{proof}

\begin{proof}[Proof of the $q=q_m$ case of Proposition \ref{lem:genidentify}] By only retaining atoms of $\pi_t$ at the initial atom locations 
  $u_i$ of $\iota(\mathbf{x})$, we can bound the LHS of \eqref{eq:genvisa} below by
  \begin{equation}\label{eq:lowerdom}\sum_{i\ge 1}\frac{1}{t}\Big(\mathbf{E}\Big[\big(\pi_t\{u_i\}\big)^{m+1}\Big]-x_i^{m+1}\Big)=\sum_{i\ge 1}\frac{1}{t}\Big(\mathbf{E}\Big[\big(W_t^{(i)}\big)^{m+1}\Big]-x_i^{m+1}\Big).
  \end{equation}
  By Lemmas \ref{lm:lower:main}--\ref{lm:upper:main} and dominated convergence we find the lower bound
    \begin{align*}&\liminf_{t\rightarrow 0+}\frac{\mathbb{E}^{\alpha,\theta}_{\mathbf{x}}\left[q_m\left({\tt RANKED}(\pi_t)\right) \right] - q_m(\mathbf{x}) }{t}\\
     &\quad\ge\sum_{i\ge 1}\left(2(m\!+\!1)(m\!-\!\alpha)x_i^m-2(m\!+\!1)(m\!+\!\theta)x_i^{m+1}\right)=2\mathcal{B}q_m(\mathbf{x}).
  \end{align*}
  For the upper bound, split $\pi_t$ as in \eqref{eq:splitpit} and bound above the sums of $(m+1)$st powers for each $\pi_t^{(i)}$, $i\ge 1$, 
  by the $(m+1)$st power of the sums $W^{(-i)}_t=\pi_t^{(i)}\big([0,1]\setminus\{u_i\}\big)$ or $X^{(i)}=\big\|\pi_t^{(i)}\big\|$, so that for all $n\ge 0$
  \begin{equation}\label{eq:qmbound}
     q_m({\tt RANKED}(\pi_t))\le\sum_{i\in[n]}\left(\big(W_t^{(i)}\big)^{m+1}+\big(W_t^{(-i)}\big)^{m+1}\right)
  					     +\sum_{i\in\mathbb{N}_0\setminus[n]}\big(X_t^{(i)}\big)^{m+1}\!.
  \end{equation}
  By Lemmas \ref{lm:lower:main}--\ref{lm:upper:negl}, this yields the upper bounds
  \begin{align*}&\limsup_{t\rightarrow 0+}\frac{\mathbb{E}^{\alpha,\theta}_{\mathbf{x}}\left[q_m\left({\tt RANKED}(\pi_t)\right) \right] - q_m(\mathbf{x}) }{t}\\
     &\quad\le\sum_{i\in[n]}\!\Big(2(m\!+\!1)(m\!-\!\alpha)x_i^m-2(m\!+\!1)(m\!+\!\theta)x_i^{m+1}\Big)\!
        +2(m\!+\!1)m\sum_{i\in\mathbb{N}\setminus[n]}x_i.
  \end{align*}
  These upper bounds converge to $2\mathcal{B}q_m(\mathbf{x})$ as $n\rightarrow\infty$, so limsup and liminf coincide and we identify the limit.  
  Note that $\mathbf{E}\big[\big(X_t^{(0)}\big)^{m+1}\big]$ does not depend on $\mathbf{x}\in\nabla_{\!\infty}$. Using Lemma \ref{lm:upper:main} on \eqref{eq:lowerdom} and \eqref{eq:qmbound} for $n\!=\!0$, we also see that
  $$\sup_{\mathbf{x}\in\nabla_{\!\infty}}
       \left|\frac{\mathbb{E}^{\alpha,\theta}_{\mathbf{x}}\left[q_m\left({\tt RANKED}(\pi_t)\right) \right] - q_m(\mathbf{x}) }{t}\right|
      \le\frac{1}{t}\mathbb{E}\big[\big(X_t^{(0)}\big)^{m+1}\big]+2(m+1)(m+\theta).$$
  By dominated convergence, \eqref{eq:genvisa} holds in $\ltwo$ with respect to ${\tt PD}(\alpha,\theta)$.
\end{proof}  
  
We now generalize this argument to all $q\!\in\!\mathcal{F}$. Let $\mathbf{m}\!=\!(m_1,\ldots,m_k)\!\in\!\mathbb{N}^{k}$, $k\!\ge\! 1$. 
We will use notation\vspace{-0.1cm} 
$$q_{\mathbf{m}}(\mathbf{x})=\prod_{j\in[k]}q_{m_j}(\mathbf{x})=\sum_{(i_1,\ldots,i_k)\in\mathbb{N}^{k}}\prod_{j\in[k]}x_{i_j}^{m_j}$$
and generalize Lemmas \ref{lm:lower:main}--\ref{lm:upper:main} to corresponding products. 

\begin{lemma}\label{lm:main1} Let $k\ge 1$, $\mathbf{m}\in\mathbb{N}^k$ and consider distinct $i_1,\ldots,i_k\ge 1$. Then 
    $$\frac{1}{t}\bigg(\mathbf{E}\bigg[\prod_{j\in[k]}\big(W^{(i_j)}_t\big)^{m_j+1}\bigg]-\prod_{j\in[k]}x_{i_j}^{m_j+1}\bigg)
          \rightarrow\mathcal{A}_kp_{\mathbf{m}}\big(x_{i_1},\ldots,x_{i_k}\big),$$
      as $t\rightarrow 0+$, where $p_{\mathbf{m}}(w_1,\ldots,w_k)=\prod_{j\in[k]}w_j^{m_j+1}$ and\vspace{-0.1cm}
      \begin{equation}\label{eq:Ak}\mathcal{A}_k:=2\!\sum_{i\in[k]}w_i\frac{\partial^2}{\partial w_i^2}-2\!\!\sum_{i,j\in[k]}w_iw_j\frac{\partial^2}{\partial w_i\partial w_j}
     -2\!\sum_{i\in[k]}\big(\theta w_i\!+\!\alpha\big)\frac{\partial}{\partial w_i}.
     \end{equation}
\end{lemma}
\begin{proof} We use Proposition \ref{prop:wfinfv} and Lemma \ref{lm:domWF}. Specifically, the quantity of interest is the 
  generator of ${\tt WF}(\mathbf{r}(k))$ with $\mathbf{r}(k)=(-\alpha,\ldots,-\alpha,\theta+k\alpha)$ applied to the function 
  $\overline{p}_{\mathbf{m}}(w_1,\ldots,w_k,w_{k+1})=p_{\mathbf{m}}(w_1,\ldots,w_k)$, and evaluated at 
  $\big(x_{i_1},\ldots,x_{i_k},1-\sum_{j\in[k]}x_{i_j}\big)$. But 
  $\mathcal{A}_{\tt WF}^{\mathbf{r}(k)}\overline{p}_\mathbf{m}(w_1,\ldots,w_k,w_{k+1})$ 
  does not depend on $w_{k+1}$ and, as a function of $(w_1,\ldots,w_k)$ coincides with $\mathcal{A}_kp_{\mathbf{m}}$.
\end{proof}

\begin{lemma}\label{lm:main2} Let $k\ge 1$ and $\mathbf{m}\in\mathbb{N}^k$. Then there is $c(\mathbf{m})>0$ such that
  \begin{align*}
    -c(\mathbf{m})\prod_{j\in[k]}x_{i_j}&\ \le\ \frac{1}{t}\bigg(\mathbf{E}\bigg[\prod_{j\in[\ell]}\big(W^{(i_j)}_t\big)^{m_j+1}\bigg]-\prod_{j\in[\ell]}x_{i_j}^{m_j+1}\bigg)\\[-0.1cm]
     &\ \le\ \frac{1}{t}\bigg(\mathbf{E}\bigg[\prod_{j\in[\ell]}\big(X^{(i_j)}_t\big)^{m_j+1}\bigg]-\prod_{j\in[\ell]}x_{i_j}^{m_j+1}\bigg)\ \le\  c(\mathbf{m})\prod_{j\in[k]}x_{i_j},
  \end{align*}
  for all $\mathbf{x}=(x_i,\,i\ge 1)\in\nabla_{\!\infty}$, all distinct $i_1,\ldots,i_k\ge 1$ and all $t>0$.\pagebreak 
\end{lemma}
\begin{proof} Let $\overline{p}_{\mathbf{m}}(w_1,\ldots,w_k,w_{k+1}):=p_{\mathbf{m}}(w_1,\ldots,w_k):=\prod_{j\in[k]}w_j^{m_j+1}$ as
  in the proof of Lemma \ref{lm:main1}. By Proposition \ref{prop:wfinfv}, 
  $$\bigg(X^{(1)}\!,\ldots,X^{(k)}\!,1\!-\!\sum_{i\in[k]}X^{(i)}\bigg)\sim{\tt WF}_{(x_{i_1},\ldots,x_{i_k},1\!-\!\sum_{i\in[k]}x_{i_j})}(0,\ldots,0,\theta).$$
  To establish the last of the claimed inequalities, rewrite the expectation as in the proof of Lemma \ref{lm:upper:main}, here using the multi-dimensional It\^o formula twice and dropping all negative terms to find an upper bound of the required form
    $$\frac{1}{t}\mathbf{E}\bigg[\int_0^t2\sum_{i\in[k]}(m_i\!+\!1)m_i p_{\mathbf{m}-e_i}\big(X^{(i_1)}_v,\ldots,X^{(i_k)}_v\big)dv\bigg]
    \le 2\sum_{i\in[k]}(m_i\!+\!1)m_i\prod_{j\in[k]}x_{i_j},$$
  where $e_i$ denotes the $i$th unit vector in $\mathbb{R}^k$. For the lower bound, the same argument applies, based on
  ${\tt WF}(-\alpha,\ldots,-\alpha,\theta+k\alpha)$ instead of ${\tt WF}(0,\ldots,0,\theta)$, here dropping all positive terms to find a similar lower bound, which allows us to choose $$c(\mathbf{m})=2\bigg(\sum_{i\in[k]}(m_i+1+\alpha)\bigg)\bigg(\sum_{j\in[k]}(m_j+1+\theta)\bigg).\vspace{-0.5cm}$$
\end{proof}

\begin{proof}[Proof of Proposition \ref{lem:genidentify}] By linearity, it suffices to consider functions
  $q(\mathbf{x})=q_{\mathbf{m}}(\mathbf{x})=\sum_{(i_1,\ldots,i_k)\in\mathbb{N}^k}\prod_{j\in[k]}x_{i_j}^{m_j}$. We use the lower bound
  $$\mathbb{E}^{\alpha,\theta}_{\mathbf{x}}\left[q_{\mathbf{m}}\left({\tt RANKED}(\pi_t)\right) \right]
     \ge\sum_{(i_1,\ldots,i_k)\in\mathbb{N}^k}\mathbf{E}\bigg[\prod_{j\in[k]}\big(W_t^{(i_j)}\big)^{m_j+1}\bigg]$$
  and adapt the proof of the case $q\!=\!q_m$ for univariate $m$. Here, we split sums according to partitions 
  $A\!=\!\{A_1,\ldots,A_r\}\!\in\!\mathcal{P}_{[k]}^{(r)}$ of $[k]$ with $r\!\in\![k]$ parts 
  \begin{equation}\label{qmsplit}q_{\mathbf{m}}(\mathbf{x})=\sum_{(i_1,\ldots,i_k)\in\mathbb{N}^k}\prod_{j\in[k]}x_{i_j}^{m_j+1}=\sum_{r\in[k]}\sum_{A\in\mathcal{P}_{[k]}^{(r)}}\sum_{\substack{h_1,\ldots,h_r\\ \text{distinct}}}\prod_{\ell\in[r]}x_{h_\ell}^{m^A_\ell+1},\end{equation}
  where $m^A_\ell\!+\!1\!:=\!\sum_{j\in A_\ell}(m_j\!+\!1)$. Then Lemma \ref{lm:main2} applies to yield for all $t\!>\!0$
  $$\frac{1}{t}\bigg|\mathbf{E}\bigg[\prod_{\ell\in[r]}\big(W_t^{(h_\ell)}\big)^{m_\ell^A+1}\bigg]-\prod_{\ell\in[r]}x_{h_\ell}^{m_\ell^A+1}\bigg|\le c(\mathbf{m}^A)\prod_{\ell\in[r]}x_{h_\ell},$$
  where $\mathbf{m}^A\!=\!(m^A_1,\ldots,m^A_r)$. The bounds are summable over distinct $h_1,\ldots,h_r$ so that we can apply dominated convergence and Lemma \ref{lm:main1} to find
  \begin{align}&\liminf_{t\rightarrow 0+}\frac{\mathbb{E}^{\alpha,\theta}_{\mathbf{x}}\left[q_{\mathbf{m}}\left({\tt RANKED}(\pi_t)\right)\right]-q_{\mathbf{m}}(\mathbf{x})}{t}\nonumber\\
  &\quad\ge\sum_{r\in[k]}\sum_{A\in\mathcal{P}_{[k]}^{(r)}}\sum_{\substack{h_1,\ldots,h_r\\ \text{distinct}}}\mathcal{A}_rp_{\mathbf{m}^A}\big(x_{h_1},\ldots,x_{h_r}\big)=2\mathcal{B}q_\mathbf{m}(\mathbf{x}).\label{eq:PetrovWF}
  \end{align}
  For the upper bounds, we use the same bounds as for \eqref{eq:qmbound}, here making sure that every $k$-tuple of atoms of $\pi_t$ is 
  taken into account, to bound  
  $\mathbb{E}^{\alpha,\theta}_{\mathbf{x}}\left[q_{\mathbf{m}}\left({\tt RANKED}(\pi_t)\right) \right]$ above by
  \begin{align}
    &\sum_{(i_1,\ldots,i_k)\in[n]^k}\mathbf{E}\bigg[\prod_{j\in[k]}\big(W_t^{(i_j)}\big)^{m_j+1}\bigg]
          +\sum_{(i_1,\ldots,i_k)\in\mathbb{N}_0^k\setminus[n]^k}\mathbf{E}\bigg[\prod_{j\in[k]}\big(X_t^{(i_j)}\big)^{m_j+1}\bigg]\nonumber\\
          &\qquad+\sum_{r\in[k]}\sum_{(i_1,\ldots,i_k)\in\mathbb{N}_0^{k}\colon i_r\in[n]}\mathbf{E}\bigg[\big(W^{(-i_r)}_t\big)^{m_r+1}\prod_{j\in[k]\setminus\{r\}}\big(X^{(i_j)}_t\big)^{m_j+1}\bigg].\label{eq:upperbound}
  \end{align}
We further bound the last term of \eqref{eq:upperbound} by $\sum_{r\in[k]}\sum_{i_r\in[n]}\mathbf{E}\big[\big(W^{(-i_r)}_t\big)^{m_r+1}\big]$. We split the middle term of 
  \eqref{eq:upperbound} as in \eqref{qmsplit}. Then Lemmas \ref{lm:upper:negl}--\ref{lm:main2} this yield the upper bounds
  \begin{align*}&\limsup_{t\rightarrow 0+}\frac{\mathbb{E}^{\alpha,\theta}_{\mathbf{x}}\left[q_{\mathbf{m}}\left({\tt RANKED}(\pi_t)\right) \right] - q_{\mathbf{m}}(\mathbf{x})}{t}\\
     &\le\sum_{r\in[k]}\sum_{A\in\mathcal{P}_{[k]}^{(r)}}\Bigg(\sum_{\substack{(h_1,\ldots,h_r)\in[n]^k\\ \text{distinct}}}\mathcal{A}_rp_{\mathbf{m}^A}(x_{h_1},\ldots,x_{h_r})
+\sum_{\substack{(h_1,\ldots,h_r)\in\mathbb{N}^k\setminus[n]^k\\ \text{distinct}}}c(\mathbf{m}^A)\prod_{\ell\in[r]}x_{h_\ell}\Bigg).
  \end{align*}
  These upper bounds converge to $2\mathcal{B}q_{\mathbf{m}}(\mathbf{x})$ as $n\rightarrow\infty$, and we conclude as in the proof of
  the case $q=q_m$ for univariate $m\in\mathbb{N}$.
\end{proof}

\subsection{Identification of $({\tt RANKED}(\pi_t),t\ge 0)$ as an ${\tt EKP}(\alpha,\theta)$}\label{sec:identification}

We now argue that our ranked process $\mathbf{V}=(\mathbf{V}_t,\,t\ge 0):=\big({\tt RANKED}(\pi_t),\,t\ge 0\big)$ is the one that is described in \cite{FengSun10}. Specifically, Feng and Sun describe their conjectured process by a Dirichlet form \cite[(3.1)]{FengSun10} on the space of all Borel probability measures on a locally compact, separable metric space (which we have, for simplicity, assumed to be $[0,1]$). We have only defined ${\tt FV}(\alpha,\theta)$ on the subspace of purely atomic measures. Our processes have a natural extension to a space that also includes probability measures whose atoms add to less than one, with the remaining mass spread uniformly over $[0,1]$, but not 
further. Here, ``natural'' means that the process is expected to immediately and continuously enter the subspace of purely atomic probability measures. 

Rather than computing the Dirichlet form of our processes, we focus on a weaker identification, showing that our processes are a ``labeled model'' that projects to the ``unlabeled model'' of \cite{FengSun10}, which \cite[Theorem 2.1]{FengSun10} identifies with ${\tt EKP}(\alpha,\theta)$. 

Let $\ltwo\left[\alpha,\theta\right]$ refer to the Hilbert space of square integrable functions on $\nabla_{\!\infty}$ with respect to the measure $\PoiDir[\alpha,\theta]$. Also, for this section, the corresponding norm will be denoted by $\norm{\cdot}_{\alpha,\theta}$. 

Let $\Prm_{\mathbf{x}}$ denote the probability measure on $\mathcal{C}\left([0,\infty),\nabla_{\!\infty} \right)$ which is the distribution of $\big({\tt RANKED}(\pi_t),\,t\ge 0\big)$ under $\BPr^{\alpha,\theta}_{\mathbf{x}}$. We will use the notation $\mathbf{V}$ for this canonical random process and notation
$\mathrm{E}_{\mathbf{x}}$ for the expectation operator.

\begin{lemma}\label{lem:strcontsemigp}
Let $\left(T_t,\; t \ge 0  \right)$ denote the transition semi-group of the process $\mathbf{V}$. Then, for every $t>0$, $T_t$ is an operator on $\ltwo\left[\alpha,\theta\right]$ and the semi-group is strongly continuous as a semi-group. 
\end{lemma}

\begin{proof} 
By definition, $T_t f(\mathbf{x})\!=\!\mathrm{E}_\mathbf{x}[ f(\mathbf{V}_t)]\!=\!\int_{\nabla_{\!\infty}} f(\mathbf{v})p_t(\mathbf{x},d\mathbf{v})$, where $p_t(\mathbf{x},d\mathbf{v})$ is the transition kernel of $\mathbf{V}$.  We first show that $T_t$ is an operator on $\ltwo\left[\alpha,\theta\right]$ in the sense that 
\begin{enumerate}[label=(\roman*), ref=(\roman*)]
\item if $f$ is square integrable with respect to $\PoiDir[\alpha,\theta]$ then so is $T_tf$,
\item if $f =0 $ $\PoiDir[\alpha,\theta]$-a.e. then so is $T_t f$.
\end{enumerate}
The second condition shows that the $\PoiDir[\alpha,\theta]$-equivalence class of $T_t f$ is determined by the $\PoiDir[\alpha,\theta]$-equivalence class of $f$, so that we may consider $T_t\colon\ltwo\left[\alpha,\theta\right] \to\ltwo\left[\alpha,\theta\right]$.  From Jensen's inequality we see that
$$ \int_{\nabla_{\!\infty}} (T_tf(\mathbf{x}))^2 \PoiDir[\alpha,\theta](d\mathbf{x}) \leq \int_{\nabla_{\!\infty}} T_tf^2(\mathbf{x})  \PoiDir[\alpha,\theta](d\mathbf{x})
     =  \int_{\nabla_{\!\infty}} f^2(\mathbf{v}) \PoiDir[\alpha,\theta](d\mathbf{v})
$$
since $\PoiDir[\alpha,\theta]$ is the stationary distribution of $\mathbf{V}$.  Both claims follow immediately.

It is easy to see that every element in the unital algebra $\mathcal{F}$ is in $\ltwo\left[\alpha,\theta\right]$. As a corollary of the $\mathbf{L}^2$ part of Proposition \ref{lem:genidentify}, $\lim_{t\rightarrow 0+} \norm{T_tq- q}_{\alpha,\theta} = 0$  for any $q \in \mathcal{F}$.  
As noted in \linebreak\cite[Section 2.2]{Petrov09}, functions in $\mathcal{F}$ have continuous extensions to the $\ell^\infty$-closure 
$\overline{\nabla}_{\!\infty}$ 
of $\nabla_{\!\infty}$, and $\mathcal{F}$ is dense in the space of bounded continuous functions on $\overline{\nabla}_{\!\infty}$, and hence also in $\mathbf{L}^2[\alpha,\theta]$.
Consider any $f \in \ltwo\left[\alpha,\theta\right]$. Then, there exists a sequence $\{q_n \} \subseteq \mathcal{F}$ such that $\lim_{n\rightarrow\infty}q_n=f$ in $\ltwo\left[\alpha,\theta\right]$. 
By the triangle inequality, 
$$\norm{(T_t-I)f}_{\alpha,\theta} \le \norm{(T_t - I)q_n}_{\alpha,\theta} + \norm{(T_t-I)\left( q_n-f \right)}_{\alpha,\theta}.$$
Since $\left(T_t-I,\; t \ge 0\right)$ is a uniformly bounded family of operators, we get $\lim_{t\rightarrow0+}T_tf = f$ in $\ltwo\left[\alpha,\theta\right]$. This proves strong continuity of the semi-group. 
\end{proof}

Hence, by \cite[Corollary 1.1.6]{EthKurtzBook}, the $\ltwo\left[\alpha,\theta\right]$ generator $\procvgen$ of $\left( T_t,\; t\ge 0 \right)$ is closed and has a dense domain in $\ltwo\left[\alpha,\theta\right]$. Moreover, by \cite[Proposition 1.2.1]{EthKurtzBook}, for any $\lambda >0$, the resolvent $\left( \lambda - \procvgen  \right)^{-1}$ exists as a bounded operator on $\ltwo\left[\alpha,\theta\right]$ and is one-to-one and has dense range. Specifically, the elementary argument of 
\cite[Step 2 in the proof of Proposition 1.4]{BoroOlsh09} gives the following result. See also \cite[Proposition 4.3]{Petrov09}.

\begin{lemma}\label{lem:resolventinverse} For any $\lambda >0$, we have $\left( \lambda - \procvgen  \right)\mathcal{F} = \mathcal{F}$.
\end{lemma}

This lemma allows us to avoid  Dirichlet form techniques while identifying our process $\mathbf{V}$ to be the one described in \cite{FengSun10}. For example, the symmetry of the resolvent $\left( \lambda - \procvgen  \right)^{-1}$ follows from the symmetry of $2\petrovgen$ on $\mathcal{F}$ (see the calculation in \cite[equation (2.5)]{FengSun10}) and Lemma \ref{lem:resolventinverse}. This shows that our process $\mathbf{V}$ is reversible with respect to $\PoiDir[\alpha,\theta]$. We skip the proof.

\begin{proof}[Step 1 of the proof of Theorem \ref{thm:petroviden}] 
Lemmas \ref{lem:strcontsemigp}--\ref{lem:resolventinverse} together with \cite[Proposition 1.3.1]{EthKurtzBook} imply that $\mathcal{F}$ is a core for $\procvgen$. 
 Letting $(\widetilde{T}_t,\; t\geq 0)$ be the $\ltwo[\alpha,\theta]$-semi-group considered by Feng and Sun \cite{FengSun10}, this shows that $(T_t,\; t\geq 0)$ and $(\widetilde{T}_{2t},\; t\geq 0)$ have the same generator (given by the closure of $(\procvgen,\mathcal{F})$) and thus are equal as semi-groups on $\ltwo[\alpha,\theta]$.
\end{proof}

Petrov \cite{Petrov09} constructs ${\tt EKP}(\alpha,\theta)$ as a Feller process on the closure\vspace{-0.1cm}
$$\overline{\nabla}_{\!\infty}
      := \bigg\{\mathbf{x}=\left(x_1, x_2, \ldots   \right)\colon x_1 \ge x_2 \ge \cdots \ge 0,\; \sum_{i\ge 1} x_i \le 1     \bigg\}\vspace{-0.1cm}$$
of our state space $\nabla_{\!\infty}$. One might wonder if our method enables construction of the Feller process (rather than its ${\bf L}^2[\alpha,\theta]$-semi-group). This is 
not so clear in our measure-valued setting, where ${\tt FV}(\alpha,\theta)$ cannot be extended to a Feller process, cf.\ \cite[below Proposition 3.6]{PaperFV}. In an interval-partition-valued setting,
we have extended corresponding processes ${\tt PDIPE}(\alpha,\theta)$, which we recall in Section \ref{sec:IP}. This yields a stronger regularity of semi-groups that allows us to
complete the proof of Theorem \ref{thm:petroviden}.   
While we could try to avoid the ${\bf L}^2$-theory in that
setting, this does not appear to save any effort. We provide some further pointers on this in Section \ref{sec:IP}.



\subsection{Fleming--Viot processes with parameters $\alpha\in(0,1)$ and $\theta\in(-\alpha,0)$}\label{sec:negtheta}

In this section we recall from \cite[Section 5.1]{ShiWinkel-2} the definition of ${\tt FV}(\alpha,\theta)$ when $\theta\in(-\alpha,0)$, and we prove Theorem \ref{thm:petroviden} (Step 1) for these processes. The main idea is simple: each atom in ${\tt SSSP}(\alpha,\theta)$ still evolves independently as ${\tt BESQ}(-2\alpha)$. New atoms are still created both as descendants of existing atoms and as some further immigration. 
However, to achieve net ``emigration'' at rate $|\theta|$, one atom does not produce descendants, and this absence of descendants is partially compensated by an immigration rate of $\theta+\alpha>0$. Here is a more formal definition.

\begin{definition}\label{def:negtheta} Let $\alpha\in(0,1)$, $\theta\in(-\alpha,0)$, $\mu\in\mathcal{M}^a$, $H_0=0$ and $\mu_0=\mu$. 
  Inductively given $(\mu_s,0\!\le \!s\!\le\! H_n)$ for some $n\!\ge\! 0$, there are two cases. If $\mu_{H_n}\!=\!0$, let $H_{n+1}\!=\!H_n$. 
  Otherwise, write $\mu_{H_n}=\sum_{i\ge 1}b_i^{(n)}\delta(u_i^{(n)})$ with $b_1^{(n)}\ge b_2^{(n)}\ge\cdots\ge 0$, consider independent 
  \[Z^{(n)}\!\sim{\tt BESQ}_{b_1^{(n)}}(-2\alpha)\quad\mbox{and}\quad\nu^{(n)}\!\sim{\tt SSSP}_{\nu_0^{(n)}}(\alpha,\theta\!+\!\alpha),\quad\mbox{where }\nu_0^{(n)}\!=\sum_{i\ge 2}b_i^{(n)}\delta(u_i^{(n)}), \]
  and set $H_{n+1}\!=\!H_n\!+\!\inf\{s\!\ge\! 0\colon Z^{(n)}_s\!=\!0\}$ and $\mu_{s}=Z^{(n)}_{s-H_n}\delta(u_1^{(n)})+\nu^{(n)}_{s-H_n}$, $s\!\in\!(H_n,H_{n+1}]$. 

  Finally, let $H_\infty=\lim_{n\rightarrow\infty}H_n$. Given $(\mu_s,\,0\le s<H_\infty)$, let $\mu_s=0$ for all $s\ge H_\infty$. We refer to $(\mu_s,\,s\ge 0)$ as an 
  ${\tt SSSP}_\mu(\alpha,\theta)$.  
\end{definition}

It was shown in \cite[Theorems 5.3 and 5.5]{ShiWinkel-2} that ${\tt SSSP}_\mu(\alpha,\theta)$ is (well-defined and) a path-continuous Hunt process, and that de-Poissonization as in and below \eqref{eq:timechangeFV} yields an $\mathcal{M}_1^a$-valued Hunt process extending ${\tt FV}(\alpha,\theta)$ to $\theta\in(-\alpha,0)$, with stationary distribution ${\tt PDRM}(\alpha,\theta)$. 
Let us revisit the main steps of our proof of Theorem \ref{thm:petroviden} in the case $\theta\in(-\alpha,0)$. 
\begin{itemize}\item Proposition \ref{prop:totalmass} holds by \cite[Theorem 5.3]{ShiWinkel-2}: ${\tt SSSP}(\alpha,\theta)$ has ${\tt BESQ}(2\theta)$ total mass process.
  \item In Proposition \ref{prop:split}, replacing $\mu^{(1)}$ by $(Z^{(0)}_s\delta(u_1^{(0)}),\,s\ge 0)$ yields a process that is ${\tt SSSP}_\mu(\alpha,\theta)$ until $Z^{(0)}$
     hits 0 and continues as ${\tt SSSP}(\alpha,\theta+\alpha)$. This is a consequence of Proposition \ref{prop:split}, applied the ${\tt SSSP}(\alpha,\theta+\alpha)$ without $\mu^{(1)}$, and of Definition \ref{def:negtheta}.
  \item Proposition \ref{prop:wfinfv} holds subject to some modifications. Specifically, due to the replacement of $\mu^{(1)}$, there is no $X^{(1)}$ here, and the Wright--Fisher part of (i) holds if $X^{(1)}$ is replaced by $W^{(1)}$, with first and last parameter changed to $-\alpha$ and $\theta+\alpha$, respectively. In (ii), there is no $W^{(-1)}$, but all claims not involving $W^{(-1)}$ continue to hold, as does Corollary \ref{cor:katoms}.
  \item Proposition \ref{prop:Markov} continues to hold. For the proof, only the third paragraph needs revisiting: a ${\tt BESQ}_{b_1}(-2\alpha)$ can replace the coupled 
    ${\tt SSSP}_{b_1\delta(u_1^\prime)}$ and ${\tt SSSP}_{b_1\delta(u_1^{\prime\prime})}$. This establishes the required coupling up to time $H_1$. An induction extends it to 
    time $H_\infty$, which suffices. 
  \item Proposition \ref{lem:genidentify} continues to hold. The proof involves several lemmas, where claims about $W^{(-1)}$ should be dropped and claims about $X^{(1)}$ 
    dropped from Lemma \ref{lm:lower:main} and replaced by $W^{(1)}$ in Lemmas \ref{lm:upper:main} and \ref{lm:main2} so that we get the upper bounds for $W^{(1)}$ instead of $X^{(1)}$. In the proof itself, 
    $W^{(-1)}$ needs to be omitted, but as its contribution was negligible asymptotically, the main change is to slightly adjust some domination bounds, also due to $\theta$ being 
    negative. We leave the details to the reader.   
  \item The remainder of the proof in Section \ref{sec:identification} holds verbatim.
\end{itemize}

\section{Continuity in the initial condition via interval partitions}\label{sec:IP}

In this section, we extend Theorem \ref{thm:petroviden} to the setting of interval partition evolutions of \cite{Paper1-2,IPPAT} and complete the proof of Theorem \ref{thm:petroviden}. 
The role of the ${\tt PDRM}(\alpha,\theta)$ stationary distributions is played by a natural two-parameter family of regenerative partitions of the unit interval $[0,1]$ that we call 
the Poisson--Dirichlet interval partitions, $\PDIP[\alpha, \theta]$; see Pitman and Winkel \cite{PitmWink09} for more details when $\theta\ge 0$. For $\theta\in(-\alpha,0)$, we 
refer to \cite{ShiWinkel-2} for a three-parameter family ${\tt PDIP}^{(\alpha)}(\theta_1,\theta_2)$ with $\theta_1,\theta_2\ge 0$ and $\theta:=\theta_1+\theta_2-\alpha\ge-\alpha$.  
For example, in order to visualize $\PDIP[\frac 12,\frac 12]$ consider a Brownian bridge during time $[0,1]$ and consider the intervals formed by the complement of the zero-set. This is distributed according to $\PDIP[\frac 12,\frac 12]$; see \cite[Example 3]{GnedPitm05}. A similar construction for Brownian motion during time $[0,1]$ gives us $\PDIP[\frac 12, 0]$; see \cite[Example 4]{GnedPitm05}. The sequence of decreasing block masses is Poisson--Dirichlet distributed. 
Replacing Brownian motion by recurrent Bessel (or ${\tt BESQ}$) processes and their bridges similarly yields $\PDIP[\alpha,0]$ and $\PDIP[\alpha,\alpha]$ for all
$\alpha\in(0,1)$.

An \emph{interval partition} is a countable set $\beta=\{J_i,i\in I\}$ of disjoint open subintervals $J_i$ of some $[0,M]$ such that the \em complement \em $C(\beta):=[0,M]\setminus\bigcup_{i\in I}J_i$ is Lebesgue-null. We write $\IPmag{\beta}$ to denote the \em total mass \em $M$. We use notation $\mathcal{I}_H$ for the set of interval partitions and equip it with the metric $d_H(\beta_1,\beta_2)$ that applies the Hausdorff metric to $C(\beta_1),C(\beta_2)\subset[0,\infty)$. 

We write $\beta_1\concat\beta_2$ for the \em concatenation \em of $\beta_1,\beta_2\in\mathcal{I}_H$ that consists of all intervals of $(c,d)\in\beta_1$, and shifted versions $(\|\beta_1\|+c,\|\beta_1\|+d)$ of all intervals $(c,d)\in\beta_2$, with similar notation $\Concat_{a\in A}\beta_a$ for the concatenation of a countable family of $\beta_a\in\mathcal{I}_H$ indexed by a totally ordered set $(A,\prec)$, with $\sum_{a\in A}\IPmag{\beta_a}<\infty$. We write $g\beta:=\{(gc,gd)\colon(c,d)\in\beta\}$ for the interval partition that has all lengths scaled by $g>0$. The empty interval partition is denoted by $\emptyset$. In analogy to \eqref{eq:Qbxralpha}, we define here on $\mathcal{I}_H$ distributions\vspace{-0.1cm}
$$\widetilde{Q}^{(\alpha)}_{b,r}:=e^{-br}\delta_\emptyset+(1-e^{-br})\mathbf{P}\Big\{\big\{\big(0,L_{b,r}^{(\alpha)}\big)\big\}\concat G\overline{\beta}\in\,\cdot\,\Big\},\vspace{-0.1cm}$$
where $G\sim{\tt Gamma}(\alpha,r)$, $\overline{\beta}\sim{\tt PDIP}(\alpha,\alpha)$ and $L_{b,r}^{(\alpha)}$ as in \eqref{eqn:LMB} are independent. In this framework, we can give the following analog of Definition \ref{def:kernel:sp}.

\begin{definition}[Transition kernel $\widetilde{K}_s^{\alpha,\theta}$] Let $\alpha\in(0,1)$, $\theta\ge 0$. For any interval partition $\beta\in\mathcal{I}_H$ and any time $s>0$, we consider the interval partition $G_0\overline{\beta}_0\concat\Concat_{J\in\beta}\Pi_J$ for independent $G_0\sim{\tt Gamma}(\theta,1/2s)$, $\overline{\beta}_0\sim{\tt PDIP}(\alpha,\theta)$ and $\Pi_J\sim\widetilde{Q}_{{\tt Leb}(J),1/2s}^{(\alpha)}$, $J\in\beta$. We denote its distribution by $\widetilde{K}_s^{\alpha,\theta}(\beta,\,\cdot\,)$.
\end{definition}

Compared with \eqref{eq:Qbxralpha} and Definition \ref{def:kernel:sp}, atom sizes such as $L_{b,r}^{(\alpha)}$ are now interval lengths (we refer to both as \em masses\em), and rather than atom locations in $[0,1]$ that were partly preserved (``survival'') partly sampled from ${\tt Unif}[0,1]$ in \eqref{eq:Qbxralpha}, we now record under $\widetilde{Q}^{(\alpha)}_{b,r}$ a left-to-right total order of intervals that places all ``descendants'' to the right of a left-most interval $\big(0,L_{b,r}^{(\alpha)}\big)$, and this order is further preserved under $\widetilde{K}_s^{\alpha,\theta}(\beta,\,\cdot\,)$ in that descendants of different ancestors inherit the order of their ancestors.

The family $\big(\widetilde{K}^{\alpha,\theta}_s\!,s\!\ge\! 0\big)$ is the transition semi-group of an $\mathcal{I}_H$-valued diffusion
$(\beta_s,s\!\ge\! 0)$\linebreak that we call ${\tt SSIPE}(\alpha,\theta)$. This was further extended in \cite[Definition 1.3]{ShiWinkel-1} to a three-parameter family ${\tt SSIPE}^{(\alpha)}(\theta_1,\theta_2)$, $\theta_1,\theta_2\!\ge\! 0$, so that $\theta\!:=\!\theta_1\!+\!\theta_2\!-\!\alpha\!\ge\!-\alpha$. 
The time-change\vspace{-0.1cm}
\begin{equation}\label{eq:timechange:IP}\rho(t)=\inf\bigg\{s\ge 0\colon\int_0^s\frac{dv}{\IPmag{\beta_v}}>t\bigg\},\vspace{-0.1cm}
\end{equation}
and normalisation to unit mass yield $\gamma_t\!:=\!\IPmag{\beta_{\rho(t)}}^{-1}\!\beta_{\rho(t)}$, $t\!\ge\! 0$. If $\beta_0\!=\!\gamma\!\in\!\mathcal{I}_H$ has total mass $\|\gamma\|\!=\!1$, we write $(\gamma_t,t\!\ge\! 0)\!\sim\!{\tt PDIPE}_\gamma(\alpha,\theta)$, respectively ${\tt PDIPE}^{(\alpha)}_\gamma(\theta_1,\theta_2)$. 
We showed in \cite[Theorems 1.3 and 1.6]{IPPAT}, \cite[Theorem 1.4]{ShiWinkel-1} and \cite[Theorem 1.4]{ShiWinkel-2}, that all of these evolutions 
are interval partition diffusions, and that ${\tt PDIPE}(\alpha,\theta)$ and ${\tt PDIPE}^{(\alpha)}(\theta_1,\theta_2)$ have ${\tt PDIP}(\alpha,\theta)$, respectively ${\tt PDIP}^{(\alpha)}(\theta_1,\theta_2)$, as their stationary distribution.

\begin{theorem}\label{IPPthm} Let $\alpha\in(0,1)$ and $\theta\ge 0$. 
  For $(\gamma_t,\,t\ge 0)\sim{\tt PDIPE}_\gamma(\alpha,\theta)$ we have 
  $\big({\tt RANKED}(\gamma_{t/2}),\,t\ge 0\big)\sim{\tt EKP}_{{\tt RANKED}(\gamma)}(\alpha,\theta)$. Similarly, for $\theta_1,\theta_2\ge 0$, $\theta:=\theta_1+\theta_2-\alpha>-\alpha$ and $(\gamma_t,\,t\ge 0)\sim{\tt PDIPE}^{(\alpha)}_\gamma(\theta_1,\theta_2)$, we have $\big({\tt RANKED}(\gamma_{t/2}),\,t\ge 0\big)\sim{\tt EKP}_{{\tt RANKED}(\gamma)}(\alpha,\theta)$.
\end{theorem}
\begin{proof}[Proof of Theorem \ref{IPPthm} and Step 2 of the proof of Theorem \ref{thm:petroviden}] 
  Let $\gamma\in\mathcal{I}_{H}$ with $\|\gamma\|=1$ and $\pi\in\mathcal{M}_1^a$ such that 
  ${\tt RANKED}(\pi)={\tt RANKED}(\gamma)$. Given that both ${\tt PDRM}(\alpha,\theta)$ and ${\tt PDIP}(\alpha,\theta)$ have 
  ${\tt PD}(\alpha,\theta)$ ranked masses, the semi-groups $K_s^{\alpha,\theta}(\pi,\cdot)$ and 
  $\widetilde{K}_s^{\alpha,\theta}(\gamma,\cdot)$, the processes ${\tt SSSP}_\pi(\alpha,\theta)$ and ${\tt SSIPE}_\gamma(\alpha,\theta)$, 
  and the processes ${\tt FV}_\pi(\alpha,\theta)$ and ${\tt PDIPE}_\gamma(\alpha,\theta)$ can be perfectly coupled so that the latter two have
  identical ranked mass processes. We remark for readers who have seen scaffolding and spindles, that this is a consequence of the clade construction in \cite{PaperFV,IPPAT} of both processes, when $\theta\ge 0$. For $\theta_1,\theta_2\ge 0$, Definition 4.1 of \cite{ShiWinkel-2} of ${\tt SSIP}^{(\alpha)}(\theta_1,\theta_2)$ similarly compares with Definition \ref{def:negtheta} of an ${\tt SSSP}(\alpha,\theta)$ with associated parameter $\theta=\theta_1+\theta_2-\alpha$ to similarly couple these processes. 
  Hence, the claims in Theorem \ref{thm:petroviden} are equivalent to the claims in Theorem \ref{IPPthm}.

  Let $\theta\ge 0$. We showed in \cite[Theorem 1.8]{IPPAT} that ${\tt PDIPE}(\alpha,\theta)$ has a Hausdorff-continuous extension to a state space of generalised interval 
  partitions of $[0,1]$ in which the requirement ${\tt Leb}([0,1]\setminus\bigcup_{i\in I} J_i)=0$ is dropped. We refer to $[0,1]\setminus\bigcup_{i\in I}J_i$ as ``dust'' and show 
  that in the generalised setting, initial dust is not negligible, but starting ${\tt PDIPE}(\alpha,\theta)$ from a state where dust has positive Lebesgue measure, the 
  evolution immediately enters $\mathcal{I}_H$ and never leaves. Indeed, we showed in \cite[Corollary 4.15]{IPPAT} that mapping the generalised ${\tt PDIPE}(\alpha,\theta)$   
  under ${\tt RANKED}$ yields a $\overline{\nabla}_{\!\infty}$-valued Feller process. In particular, the projected $\nabla_{\!\infty}$-valued ${\tt PDIPE}(\alpha,\theta)$ itself, $\mathbf{V}_t = {\tt RANKED}(\gamma_t)$, $t\ge0$, is continuous in the initial state for 
  initial states in $\nabla_{\!\infty}$. For ${\tt PDIPE}^{(\alpha)}(\theta_1,\theta_2)$, the corresponding continuity in the initial state was obtained in \cite[Theorem 1.4]{ShiWinkel-2}.

  Now recall that Step 1 of the proof of Theorem \ref{thm:petroviden} yields the identification of $\mathbf{L}^2[\alpha,\theta]$-semi-groups, $(T_t,\,t\ge0)$ and $(\widetilde{T}_{2t},\,t\ge0)$. 
  Let $(\widetilde{\mathbf{V}}_t, {t\geq 0})$ be the diffusion associated with $(\widetilde{T}_{t},\; t\geq 0)$ constructed in \cite{Petrov09} (which is the Feller version of the diffusion constructed in \cite{FengSun10}). Let $\widetilde{\mathrm{P}}_\mathbf{x}$ denote the law of $\widetilde{\mathbf{V}}$, when starting from $\mathbf{x}$. Then we find that for every $f\in \ltwo\left[\alpha,\theta\right]$ we have 
  \[ 
    \mathrm{E}_\mathbf{x}\left[f(\mathbf{V}_t)\right] = T_tf(\mathbf{x}) = \widetilde{T}_{2t} f(\mathbf{x}) = \widetilde{\mathrm{E}}_\mathbf{x}[f(\widetilde{\mathbf{V}}_{2t})] \qquad \textrm{for }\PoiDir[\alpha,\theta]\textrm{-a.e. }\mathbf{x}\in\nabla_{\!\infty}.
  \] 
  Now consider $f\colon\nabla_{\!\infty}\rightarrow[0,\infty)$ bounded and continuous. Then 
  $\mathbf{x}\mapsto\mathrm{E}_\mathbf{x} [f(\mathbf{V}_t)]$ is continuous, and $\mathbf{x}\mapsto  \widetilde{\mathrm{E}}_\mathbf{x}[f(\widetilde{\mathbf{V}}_{2t})]$ is 
  continuous by \cite[Proposition 4.3]{Petrov09}. As any set of full $\PoiDir[\alpha,\theta]$-measure is dense in $\nabla_{\!\infty}$, we get 
  $\mathrm{E}_{\mathbf{x}} [f(\mathbf{V}_t)]\!=\!\widetilde{\mathrm{E}}_{\mathbf{x}}[f(\widetilde{\mathbf{V}}_{2t})]$ for every bounded, continuous $f$ and every 
  $\mathbf{x}\in\nabla_{\!\infty}$. Together with path-continuity and the Markov property, this identifies the laws of the processes $(\mathbf{V}_t,\,t\ge 0)$ and $(\widetilde{\mathbf{V}}_{2t},\,t\ge 0)$. 
\end{proof}

Since we showed in \cite[Corollary 4.15]{IPPAT} that for $\theta\ge0$, mapping the generalised ${\tt PDIPE}(\alpha,\theta)$ under ${\tt RANKED}$ yields a $\overline{\nabla}_{\!\infty}$-valued Feller process, Theorem \ref{IPPthm} also identifies this Feller process with Petrov's Feller version of ${\tt EKP}(\alpha,\theta)$. 
Returning to the question of whether this
allows one to avoid the ${\bf L}^2$-theory used in Section \ref{sec:identification}, the answer is yes, but at a cost, as this would require further estimates of the type 
established in Lemmas \ref{lm:upper:negl} and \ref{lm:main2} to handle the (mostly negligible) contribution of dust to the pre-generator on $\mathcal{F}$. We omit the details.

\bibliographystyle{abbrv}
\bibliography{AldousDiffusion4}
\end{document}